\documentclass[12pt]{article}
\usepackage{amsmath}
\usepackage{amsmath}
\usepackage{amsfonts}
\usepackage{amsthm}
\usepackage{amssymb}
\usepackage{graphicx,psfrag,epsf}
\usepackage{enumerate}
\usepackage{natbib}
\usepackage[table]{xcolor}
\usepackage{url} % not crucial - just used below for the URL 

\newcommand{\unblind}{1}

\newcommand{\bo}[1]{\boldsymbol{#1}}
\newtheorem{theorem}{Theorem}[section]
\newtheorem{definition}{Definition}[section]
\newtheorem{remark}{Remark}[section]
\newtheorem{lemma}{Lemma}[section]
\newtheorem{assumption}{Assumption}[section]

\def\red#1 {{\color{red}{{#1}}}\ }
\def\klaus#1 {{\color{blue}{{#1}}}\ }
\def\ca#1{{\cal #1}}
\def\nano{\scriptscriptstyle}
\def\lo#1{_{  #1}}
\def\hi#1{^{  #1}}
\def\op{\nano{\rm \uppercase{op}}}

\def\tsum{\textstyle{\sum}}
\def\ca{\mathcal}
\def\real{{\mathbb R}}
\def\cran{\overline{\rm{ran}}}
\def\oc{\hi \perp}
\def\ker{{\rm{ker}}}
\def\var{{\rm{var}}}
\def\hii#1{\hi { #1 }}

\newcommand\numberthis{\addtocounter{equation}{1}\tag{\theequation}}

\usepackage{authblk}

\begin{document}

\def\spacingset#1{\renewcommand{\baselinestretch}%
{#1}\small\normalsize} \spacingset{1}

%%%%%%%%%%%%%%%%%%%%%%%%%%%%%%%%%%%%%%%%%%%%%%%%%%%%%%%%%%%%%%%%%%%%%%%%%%%%%%

\if1\unblind
{
  \title{\bf Independent component analysis for multivariate functional data\thanks{The research of Joni Virta and Hannu Oja was partially supported by the Academy of Finland Grant 268703. The research of Bing Li was supported in part by the U.S. National Science Foundation grants DMS-1407537 and DMS-1713078. The research of Klaus Nordhausen was supported by CRoNoS COST Action IC1408.
}\hspace{.2cm}}
   \author[1]{Joni Virta}
   \author[2]{Bing Li}
   \author[3]{Klaus Nordhausen}
   \author[1]{Hannu Oja}
   \affil[1]{University of Turku, Finland}
   \affil[2]{Pennsylvania State University, PA, U.S.A.}
   \affil[3]{Vienna University of Technology, Austria}
   \date{}
  \maketitle
} \fi

\if0\unblind
{
  \bigskip
  \bigskip
  \bigskip
  \begin{center}
    {\LARGE\bf Independent component analysis for multivariate functional data}
\end{center}
  \medskip
} \fi

%\bigskip
\begin{abstract}
We extend two methods of independent component analysis, fourth order blind identification and joint approximate diagonalization of eigen-matrices, to vector-valued functional data. Multivariate functional data occur naturally and frequently in modern applications, and extending independent component analysis to this setting allows us to distill important information from this type of data, going a step further than the functional principal component analysis. To allow the inversion of the covariance operator we make the assumption that the dependency between the component functions lies in a finite-dimensional subspace. In this subspace we define fourth cross-cumulant operators and use them to construct the two novel, Fisher consistent methods for solving the independent component problem for vector-valued functions. Both simulations and an application on a hand gesture data set show the usefulness and advantages of the proposed methods over functional principal component analysis.
\end{abstract}

\noindent%
\vfill

\newpage
%\spacingset{1.45} % DON'T change the spacing!
\section{Introduction}

\subsection{Independent component analysis}

Independent component analysis is a classical problem in multivariate statistics and signal processing where one assumes that the observed independent and identically distributed random vectors are linear mixtures of latent random vectors having independent marginal distributions. At its simplest this corresponds to presuming that, given the observed random vector $\textbf{x} \in \mathbb{R}^p$, there exists a non-singular unmixing matrix $\bo{\Gamma} \in \mathbb{R}^{p \times p}$ such that
\begin{equation}\label{eq:icm_vector}
\bo{\Gamma} \textbf{x} = \textbf{z},
\end{equation}
where the random vector $\textbf{z} \in \mathbb{R}^p$ has independent marginals. In the independent component problem a random sample of $\textbf{x}$ is observed and the objective is to estimate any matrix $\bo{\Gamma}$ such that \eqref{eq:icm_vector} holds. We say any matrix as the formulation of the problem is clearly not well-defined, one can freely scale, permute and change the signs of the rows in \eqref{eq:icm_vector} and the right-hand side still retains the independence of its components. As such the constraint $\mbox{cov}(\textbf{z}) = \textbf{I}$ is usually introduced, freeing us of the scale invariance. Further assuming that at most one of the components of $\textbf{z}$ is normally distributed, one can show that the vector $\textbf{z}$ can be estimated up to marginal signs, order and location \citep{comon2010handbook}.

Since its introduction in the 1980s a multitude of methods with varying approaches and assumptions have been proposed for solving the problem. These methods are generally based either on projection pursuit, decompositions of various matrices of cumulants or maximum likelihood. The most well-known example belonging to the first class is FastICA \citep{hyvarinen1997fast}, a projection pursuit method that extracts the independent components either sequentially or simultaneously by maximizing some measure of non-Gaussianity. Several different variations of FastICA exist, see for example \cite{koldovsky2006efficient, miettinen2014deflation, miettinen2017squared}. The second class includes classic methods like FOBI and JADE (see below) but also several newer ones such as \cite{moreau2001generalization, bonhomme2009consistent}. For an example of likelihood-based methods, see e.g. \cite{risk2015likelihood}.

In this work we focus exclusively on two of the very first methods proposed for independent component analysis, fourth order blind identification (FOBI) \citep{cardoso1989source} and joint approximate diagonalization of eigenmatrices (JADE) \citep{cardoso1993blind}, which are simply based on the diagonalization of various moment-based matrices. As such FOBI and JADE offer an easy starting point for various extensions of independent component analysis into the realms of non-standard data structures, some examples including versions specially tailored for time series data \citep{matilainen2015new}, tensor-valued data \citep{virta2016independent, virta2016jade} and univariate functional data \citep{li2015functional}.

Before describing our contribution we first briefly review the key steps behind FOBI and JADE, both to motivate our exposition and to contrast the constructions in the later sections. Namely, in both methods we assume that the zero-mean random vector $\textbf{x} \in \mathbb{R}^p$ obeys the independent component model in \eqref{eq:icm_vector}, and additionally that the components of $\textbf{z}$ have finite fourth moments, $\beta_i = E(z_i^4) < \infty$, $i=1, \ldots , p$. A basic result in independent component analysis then says that if $\bo{\Sigma}(\textbf{x})^{-1/2}$ is the symmetric inverse square root of the covariance matrix of $\textbf{x}$, then there exists an orthogonal matrix $\textbf{U} \in \mathbb{R}^{p \times p}$ such that the standardized random vector satisfies $\textbf{x}_{st} = \bo{\Sigma}(\textbf{x})^{-1/2} \textbf{x} = \textbf{U} \textbf{z}$.

For estimating the unknown matrix $\textbf{U}$ both methods utilize fourth moments. Defining next the matrices
\begin{equation}\label{eq:cumulant_matrix}
\textbf{C}^{ij}(\textbf{x}_{st}) = E \left\{ (\textbf{x}_{st}^T \textbf{e}_i) (\textbf{x}_{st}^T \textbf{e}_j) \textbf{x}_{st} \textbf{x}_{st}^T \right\} - \delta_{ij} \textbf{I} - \textbf{e}_i \textbf{e}_j^T - \textbf{e}_j \textbf{e}_i^T,
\end{equation}
where $\textbf{e}_{i}$ is the $i$th member of the canonical basis of $\mathbb{R}^p$ --- that is, $\textbf{e} \lo i$ has all components equal to 0 except its $i$th component, which is 1, and $\delta_{ij}$ is the Kronecker delta. The set $\mathcal{C} = \{ \textbf{C}^{ij}(\textbf{x}_{st}) \mid i,j=1, \ldots , p \}$ collects every fourth cross-cumulant of the standardized random vector $\textbf{x}_{st}$. It can be shown that under the model the unknown orthogonal matrix $\textbf{U}^T$ diagonalizes all matrices in the set $\mathcal{C}$ and JADE estimates $\textbf{U}^T$ by simultaneously (approximately) diagonalizing these matrices. FOBI can be viewed as a lighter version of JADE in that it only diagonalizes the single matrix $\sum_{i=1}^p \textbf{C}^{ii}(\textbf{x}_{st}) = E ( \textbf{x}_{st}\textbf{x}_{st}^T \textbf{x}_{st} \textbf{x}_{st}^T ) - (p + 2) \textbf{I}$, which is the sum of a subset of members of $\mathcal{C}$. By this heuristic it seems reasonable to speculate that JADE outperforms FOBI, which indeed is generally the case: see, for example, \cite{miettinen2015fourth}. Additionally, for JADE to be Fisher consistent it is sufficient that at most one of the $\beta \lo i$'s  is zero, whereas for FOBI to be  Fisher consistent we need  the stronger condition that all  $\beta_i$  are distinct. However, JADE pays for its advantages by being computationally much heavier than FOBI, and when a quick application of an independent component analysis method is needed, FOBI is often the first choice.

\subsection{Independent component analysis and functional data}

As the main contribution of this work we further extend on the functional independent component analysis proposed in \cite{li2015functional} by considering not real-valued functions but instead functions that take values in the $p$-dimensional Euclidean space. That is, for each observational unit we observe $p$ functions not necessarily residing in the same function space. Data of this form is increasingly common nowadays and some areas of application include electroencephalography (EEG) data, socio-economic time series data observed for multiple areas/countries and three-dimensional location data measured for multiple observational units over time.

Although univariate functional data analysis is currently exceedingly popular, its multivariate counterpart has received relatively little attention in the literature. Some previous contributions to the field include: \cite{silverman2005functional,berrendero2011principal,
sato2013theoretical,chiou2014multivariate,jacques2014model,
happ2016multivariate} discussed multivariate functional principal component analysis, \cite{jacques2014model} using the extracted principal components to conduct clustering and \cite{happ2016multivariate} allowing different domains for the component functions; \cite{tokushige2007crisp,ieva2011ecg} developed multivariate functional clustering using \textit{k}-means and \cite{kayano2010functional} used orthonormalized Gaussian basis functions for the same purpose; \cite{li2016nonlinear} developed sufficient dimension reduction methodology where both the predictor and the response can be multivariate functional data.

Consider next the conceptual and theoretical differences between multivariate-functional and univariate-functional extensions of independent component analysis. The two key aspects of independent component analysis are statistical independence and the notion of marginals. In a sense, the multivariate functional extension considered here is conceptually easier than the univariate functional extension developed in \citet{li2015functional}. As observed in that paper, unlike in the classical setting, the univariate functional data do not have natural marginal random variables on which to perform independent component analysis. \citet{li2015functional} tackled this issue by using the coefficients in the Karhunen-Loeve expansion as the marginal random variables to prompt the process.
Independent components are then defined in terms of these coefficients, see also \citet{gutch2012infinity}. For multivariate functional data, however, we can take a more straightforward route of simply treating the component functions as the marginals. In this context the independent component problem has the intuitively appealing objective of, given an observed multivariate random function, trying to extract another multivariate random function with independent component functions. These independent component functions can then be various latent processes, such as vital signs in the context of EEG-data. A finite-dimensional analogue for our problem is the independent subspace analysis \citep{cardoso1998multidimensional}, where we try to divide a larger space into a collection of smaller, independent subspaces. To sum up, the independent components in \citet{li2015functional} are random variables, but the independent components in this paper are random functions. From this perspective, this paper is not an extension of \citet{li2015functional}, but instead an extension of the classical independent component analysis into a different direction.

In Section \ref{sec:theoretical} we go briefly through the basics of functional analysis. The section also introduces the Cartesian product space $\mathcal{H}$ where our observed functions will reside in and a natural subclass of linear operators therein. Section \ref{sec:probability} equips the space $\mathcal{H}$ with a suitable probability structure and, having defined what we mean by a random multivariate function $X \in \mathcal{H}$, defines the covariance matrix operator of $X$. The proposed methods of functional independent component analysis are described in Section \ref{sec:independent} along with a proof of their Fisher consistency. In Section \ref{sec:sample} we derive the coordinate representations for the sample versions of the methods and in Section \ref{sec:examples} use them in a simulation study and in an application on the \textit{uWave} hand gesture data set \citep{liu2009uwave}. Finally, we close in Section \ref{sec:discussion} with some discussion and prospective ideas. The simulation and real data example were conducted with R \citep{Rcore} using the packages fda \citep{Rfda}, ggplot2 \citep{Rggplot2}, JADE \citep{MiettinenNordhausenTaskinen2017},  MASS \citep{Rmass} and reshape2 \citep{Rreshape}.

\section{Theoretical framework}\label{sec:theoretical}

\subsection{The Hilbert space $\mathcal{H}$ of vector-valued functions}

We next review the basics of functional analysis, see \cite{conway2013course} for a standard treatment. Let $T \subset \mathbb{R}$ be an interval and $( \mathcal{H}_i, \langle \cdot , \cdot \rangle_i )$, $i = 1, \ldots , p$, be separable Hilbert spaces of functions from $T$ to $\mathbb{R}$. Furthermore, let $\mathcal{B}_i$ be the Borel $\sigma$-field generated by the open sets in $\mathcal{H}_i$ with respect to the metric induced by $\langle \cdot , \cdot \rangle_i$. Let $\ca H$ be the direct sum of $\ca H \lo 1, \ldots, \ca H \lo p$; that is, $\mathcal{H} = \times_{i=1}^p \mathcal{H}_i$ is the Cartesian product of the individual spaces and  the inner product in $\ca H$ is defined by $\langle f , g \rangle_\mathcal{H} = \langle f_1 , g_1 \rangle_1 + \cdots + \langle f_p , g_p \rangle_p$, for any $f = (f_1, \ldots , f_p) \in \ca H$ and $g = (g_1, \ldots , g_p) \in \ca H$. Denoting the norms induced by the inner products $\langle \cdot , \cdot \rangle_\mathcal{H}$, $\langle \cdot , \cdot \rangle_1$, \ldots , $\langle \cdot , \cdot \rangle_p$ by $\| \cdot \|_\mathcal{H}$, $\| \cdot \|_1$, \ldots , $\| \cdot \|_p$, respectively, the relation $\| f \|^2_\mathcal{H} = \| f_1 \|^2_1 + \cdots + \| f_p \|^2_p$ is easily seen to hold for any $f = ( f \lo 1, \ldots, f \lo p ) \in \mathcal{H}$. Furthermore, a natural $\sigma$-field in $\mathcal{H}$ is the product $\sigma$-field $\mathcal{B} = \mathcal{B}_1 \times \cdots \times \mathcal{B}_p$ generated by all measurable rectangles $B_1 \times \cdots \times B_p$ where $B_i \in \mathcal{B}_i$, $i=1, \ldots , p$.

Being separable, each $\mathcal{H}_i$ admits a countable orthonormal basis, $\mathcal{E}_i = \{ e_{ik} \}_{k=1}^\infty$. Using the component bases we construct an orthonormal basis in $\ca H$ as follows. Let $e^+_{ik}$ denote the $p$-dimensional  vector of functions whose components are 0 except for the $i$th component, which is  $e \lo {ik}$. Then $\{ e \lo {ik} \hi +: i = 1, \ldots, p, k = 1, 2, \ldots \}$ is an orthonormal basis  of $\ca H$. This construction implies that the product space $\mathcal{H}$ is also separable. Throughout the paper any vector $f \in \mathcal{H}$ which has exactly one non-zero component will be called canonical, in relation to such a vector's resemblance to the canonical basis vectors in the Euclidean spaces.

Let $\mathcal{L}(\mathcal{H}_j, \mathcal{H}_i)$ be the set of all bounded linear operators from $\mathcal{H}_j$ to $\mathcal{H}_i$. That is, a linear operator $L_{ij}$ is in $\mathcal{L}(\mathcal{H}_j, \mathcal{H}_i)$ if and only if there exists a positive $M_{ij}$ such that for all $f_j \in \mathcal{H}_j$ we have $\| L_{ij} f_j \|_i \leq M_{ij} \| f_j \|_j$. Then for any $i, j$, $(\mathcal{L}(\mathcal{H}_j, \mathcal{H}_i), \| \cdot \|_{\op, ij})$ is a Banach space where the operator norm $\| \cdot \|_{\op, ij}$ is defined as
\begin{equation*}
\| L_{ij} \|_{\op, ij} = \underset{f_j \neq 0}{\mbox{sup}} \left( \frac{\| L_{ij} f_j \|_i}{\| f_j \|_j} \right).
\end{equation*}
In the following we will use the notation $\| \cdot \|\lo {\op}$ for all possible operator norms and the context will always make clear which operator norm we mean. Similarly, $I$ will be used to denote the identity operator of all considered spaces, the context again making the intended use clear. Recall also that for all $L_{ij} \in \mathcal{L}(\mathcal{H}_j, \mathcal{H}_i)$, there exists the adjoint operator $L_{ij}^*$, defined as the unique member of $\mathcal{L}(\mathcal{H}_i, \mathcal{H}_j)$ that satisfies $\langle L_{ij} f_j ,  f_i \rangle_i = \langle f_j , L_{ij}^* f_i \rangle_j$, for all $f_i \in \mathcal{H}_i$ and $f_j \in \mathcal{H}_j$.

Finally, define the tensor product $f_i \otimes f_j$ of $f_i \in \mathcal{H}_i$ and $f_j \in \mathcal{H}_j$ as the linear operator from $\mathcal{H}_j$ to $\mathcal{H}_i$ having the action $g_j \mapsto \langle f_j , g_j \rangle_j f_i$. Equivalent properties to those listed for tensor product operators from $\mathcal{H}$ to $\mathcal{H}$ in Lemma 2 of \cite{li2015functional} can also be proven for the tensor product operators from $\mathcal{H}_j$ to $\mathcal{H}_i$.

\subsection{Matrices of bounded linear operators in $\mathcal{H}$}

We next consider a natural subset of the set of all bounded linear operators from $\mathcal{H}$ to $\mathcal{H}$, constructed using bounded linear operators from the component spaces to each other. For a set of operators $\{L \lo {ij} \in \ca L ( \ca H \lo j, \ca H \lo i ): \, i,j=1, \ldots, p \}$, let   $L $ be the operator
\begin{align}\label{eq:additive mapping}
\ca H \to \ca H, \quad  (f \lo 1, \ldots, f \lo p) \in \ca H \mapsto \left(\tsum \lo {j=1} \hi p L \lo {1j} f \lo j, \ldots, \tsum \lo {j=1} \hi p L \lo {pj} f \lo j \right).
\end{align}
Intuitively, we can identify  $L$ with the matrix of bounded linear operators,
\begin{equation*}
L \equiv
\begin{pmatrix}
L_{11} & \cdots & L_{1p} \\
\vdots & \ddots & \vdots \\
L_{p1} & \cdots & L_{pp}
\end{pmatrix},
\end{equation*}
so that the map in (\ref{eq:additive mapping}) can be formally regarded as matrix multiplication. We denote the class of all such operators as $\mathcal{L}(\mathcal{H}) = \times_{i,j=1}^p\mathcal{L}(\mathcal{H}_j, \mathcal{H}_i)$. The same construction  was used in \cite{li2016}. See also \cite{sato2013theoretical} and \cite{li-chun-zhao-2014}.

Using \eqref{eq:additive mapping}, it is easy to check that an operator   $L \in \ca L ( \ca H)$ is also linear;  that is, $L(af + bg)  = a ( Lf ) + b ( Lg )$, for all $f,g \in \mathcal{H}$ and $a,b \in \mathbb{R}$. Furthermore, using the Cauchy-Schwarz inequality and the operator norm inequality one can show that, for all $L \in \mathcal{L}(\mathcal{H})$ and $f \in \mathcal{H}$, we have
\begin{equation*}
\| Lf \| \lo {\mathcal{H}} \leq \left( \tsum_{i,j = 1}^p \| L_{ij} \|^2_{\op} \right)^{1/2} \| f \| \lo {\mathcal{H}}.
\end{equation*}
That is, an operator $L \in \mathcal{L}(\mathcal{H})$ is also bounded. Thus, an operator $L \in \ca L ( \ca H)$ inherits both linearity and boundedness from its component operators $L_{ij}$. Consequently, being a bounded linear operator, any $L$ $\in$ $\mathcal{L}(\mathcal{H})$ admits the adjoint operator $L^*$. Using some algebra it is easily seen that the elements of the adjoint satisfy $(L^*)_{ij} = L_{ji}^*$, drawing an analogy to the Hermitian adjoint of a matrix in $\mathbb{C}^{p \times p}$.

Two useful subsets of $\mathcal{L}(\mathcal{H})$ are now readily defined. Call a member $L \in \mathcal{L}(\mathcal{H})$ a diagonal matrix of operators (or simply diagonal) if $L_{ij} = 0$ whenever $i \neq j$ and $L_{ii}^* = L_{ii}$. The simplest diagonal operator is the identity operator for which $L_{ii} = I$, $i = 1, \ldots ,p$. Diagonal operators play later a central role in estimating solutions to the functional independent component model and as one of our key results we prove in Section \ref{sec:independent} a connection between diagonal operators and canonical vectors. Finally, an element $U \in \mathcal{L}(\mathcal{H})$ is called unitary if $U^* U = U U^* = I$. Using the component representation it is easily seen that a sufficient and necessary condition for $U$ to be unitary is
\begin{equation*}
\sum_{k=1}^p U_{ik} U_{jk}^* = \Delta_{ij}  , \quad \mbox{for all } i,j=1, \ldots , p,
\end{equation*}
where $\Delta \lo {ij}$ is the zero operator if $i \ne j$, and is the identity operator from $\ca H \lo i$ to $\ca H \lo j$ if $i = j$. This is
a clear analogy for the orthonormality of the rows of a unitary matrix in $\mathbb{C}^{p \times p}$.

\section{Probability structure on $\mathcal{H}$}\label{sec:probability}

\subsection{Random elements in $\mathcal{H}$}

Let $(\Omega, \mathcal{F}, \mathbb{P})$ be a probability space. A random element in $\mathcal{H}_i$ is a function $X_i: \Omega \rightarrow \mathcal{H}_i$ that is $\mathcal{F} / \mathcal{B}_i$-measurable, $i = 1, \ldots , p$. Similarly, a random element in $\mathcal{H}$ is a function $X: \Omega \rightarrow \mathcal{H}$ that is $\mathcal{F} / \mathcal{B}$-measurable. A random element $X$ in $\mathcal{H}$ can thus be thought of as a random function $X(\cdot) = (X_1(\cdot), \ldots X_p(\cdot))$, where $X \lo i$ resides in $\ca H \lo i$, $i = 1, \ldots, p$. For the basic theory of random variables in function spaces see \cite{bosq2012linear}.

In the following, we  denote the set of all $m$th power integrable random elements in $\mathcal{H}$ by $\ca X \hi m ( \ca H)$, that is,
\begin{align*}
\mathcal{X}^m(\mathcal{H}) = \{ ( X: \Omega \rightarrow \mathcal{H} ):  \ E \left( \| X \|_\mathcal{H}^m \right) < \infty \}.
\end{align*}
It is easily seen that requiring $X \in \mathcal{X}^2(\mathcal{H})$ or $X \in \mathcal{X}^4(\mathcal{H})$ is equivalent to requiring the component functions to respectively satisfy $X_i \in \mathcal{X}^2(\mathcal{H}_i)$ or $X_i \in \mathcal{X}^4(\mathcal{H}_i)$, for all $i=1, \ldots , p$.

Next, define a random operator $W_{ij}$ to be a mapping $W_{ij}: \Omega \rightarrow \mathcal{L}(\mathcal{H}_j, \mathcal{H}_i)$ that is $\mathcal{F}/\mathcal{B}_{\op}$-measurable where $\mathcal{B}_{\op}$ is the Borel $\sigma$-field generated by the open sets of $\mathcal{L}(\mathcal{H}_j, \mathcal{H}_i)$ with respect to the metric induced by the operator norm $\| \cdot \|_{\op}$. If $W_{ij}$ is a random operator with $E \left( \| W_{ij} \|\lo {\op} \right) < \infty$, then the bivariate map $(f_i, f_j) \mapsto E \left( \langle f_i , W_{ij} f_j \rangle_i \right)$ is a bounded bilinear form and can be shown to induce a unique operator $A_{ij} \in \mathcal{L}(\mathcal{H}_j, \mathcal{H}_i)$ satisfying $\langle f_i , A_{ij} g_j \rangle_i = E \langle f_i , W_{ij} g_j \rangle_i$ for all $f_i \in \mathcal{H}_i$ and $g_j \in \mathcal{H}_j$. We define the expected value of $W_{ij}$ to be this operator, $E \left( W_{ij}\right) = A_{ij}$

Using the previous we are now sufficiently equipped to define the first two moments, the mean function and the covariance matrix operator, of a random element $X \in \mathcal{H}$.

\subsection{The covariance matrix operator $\Sigma_{X X}$}

Assume next that $X \in \mathcal{X}^2 (\mathcal{H})$. The expected values $E(X_i) = \mu_i \in \mathcal{H}_i$, $i = 1, \ldots , p$, are readily defined as the Riesz representation of the bounded linear functional
\begin{align*}
\ca H \lo i \to \real, \quad f \lo i \mapsto E \left( \langle f \lo i, X \lo i \rangle \lo { i} \right).
\end{align*}
Using the component-wise expected values $\mu_i$ we further define the expected value of the random element $X$ to be the function $\mu = (\mu_1, \ldots , \mu_p) \in \mathcal{H}$. As we can always center our observed data, it is not restricting to assume that $\mu = 0$, as we will do for the remainder of this work.

Consider then the random operator $(X_i \otimes X_j)$. Using the Cauchy-Schwarz inequality we have
\begin{equation*}
E \left( \| X_i \otimes X_j \|\lo {\op} \right) \leq \left\{ E \left( \| X_i \|_i^2 \right) E \left( \| X_j \|_j^2 \right) \right\}\hi {1/2},
\end{equation*}
the right-hand side of which is finite due to our assumption on square integrability. The random operator $(X_i \otimes X_j)$ thus induces the unique, bounded linear operator, $\Sigma \lo {X \lo i X \lo j} = E(X_i \otimes X_j)$, the cross-covariance operator \citep{baker1973joint} between $X_i$ and $X_j$. Using the definition of the expected value of a random operator one can further show that the adjoint operator of $\Sigma_{X \lo i X \lo j}$ is $\Sigma \lo {X \lo i X \lo j}^* = \Sigma \lo {X \lo j X \lo i}$. %, whose action on $f_j$ satisfies \citep{hsing2015theoretical}
%\begin{equation}\label{eq:S_action}
%E \left(X_i \otimes X_j \right) f_j = E \left\{ \left( X_i \otimes X_j \right) f_j \right\}.
%\end{equation}

Using the $p^2$ bounded linear operators $\Sigma \lo {X \lo i X \lo j}$  we next construct the covariance matrix operator $\Sigma \lo {XX} \in \mathcal{L}(\mathcal{H})$ as
\begin{equation*}
\Sigma \lo {XX} \equiv
\begin{pmatrix}
\Sigma \lo {X \lo 1 X \lo 1} & \cdots & \Sigma \lo {X \lo 1 X \lo p} \\
\vdots & \ddots & \vdots \\
\Sigma \lo {X \lo p X \lo 1} & \cdots & \Sigma \lo {X \lo p X \lo p}
\end{pmatrix},
\end{equation*}
It is easily seen that, for $f = (f_1, \ldots , f_p) \in \mathcal{H}$, we have the equality $ f \otimes f  = ( f_i \otimes f_j )_{i,j=1}^p$ and the covariance matrix operator $\Sigma_{XX}$ can then be written compactly as $E (X \otimes X )$. This type of matrices of covariance operators were also used in \cite{li-song-2017} and \cite{song-li-2017}.

\begin{remark}
For clarity we use two different notations for the covariance matrix operator of a random function $X \in \mathcal{X}^2(\mathcal{H})$: when it is understood as a bounded linear operator in $\mathcal{H}$ we use the notation $\Sigma \lo {XX}$; when it is understood as the mapping $\mathcal{X}^2(\mathcal{H}) \rightarrow \mathcal{L}(\mathcal{H}), \quad X \mapsto \Sigma \lo {XX},$
we use the notation $\Sigma$.
\end{remark}

Recall next four key properties of the ordinary covariance matrix $\mbox{cov}(\textbf{x})$ of a square-integrable random vector $\textbf{x} = (x_1, \ldots , x_p)^T$: i) self-adjointness (symmetry), $\mbox{cov}(\textbf{x}) = \mbox{cov}(\textbf{x})^T$, ii) positive-semidefiniteness, for any $\textbf{a} \in \mathbb{R}^p$ we have $\textbf{a}^T \mbox{cov}(\textbf{x}) \textbf{a} \geq 0$, iii) affine equivariance, for any invertible matrix $\textbf{A} \in \mathbb{R}^{p \times p}$ the covariance matrix transforms as $\mbox{cov}(\textbf{Ax}) = \textbf{A} \mbox{cov}(\textbf{x}) \textbf{A}^T$ and iv) full independence property, if $x_i$ and $x_j$ are independent then $\ \mbox{cov}(\textbf{x}) _{ij} = 0$. Not surprisingly, it turns out that all of these properties are shared also by the covariance matrix operator $\Sigma_{X X}$, as described in the following lemma.

\begin{lemma} \label{lem:S_properties}
Assuming $X \in \mathcal{X}^2(\mathcal{H})$, the covariance matrix operator $\Sigma_{X X} \in \mathcal{L}(\mathcal{H})$ has the following properties:
\begin{itemize}
\item[i)] It is a self-adjoint, non-negative, trace-class operator and as such admits a spectral decomposition with the associated orthonormal basis $\{ \phi_k \}_{k=1}^\infty$.
\item[ii)] As a mapping $\Sigma: \mathcal{X}^2(\mathcal{H}) \rightarrow \mathcal{L}(\mathcal{H})$, the covariance matrix operator is affine equivariant in the sense that $\Sigma(AX) = A \Sigma(X) A^*$ for any invertible bounded linear operator $A \in \mathcal{L}(\mathcal{H})$.
\item[iii)] If $X_i$ and $X_j$ are independent, $\Sigma \lo {X \lo i X \lo j} = 0$.
\end{itemize}
\end{lemma}

These properties were established in \cite{li2015functional} for the case of univariate $X$.

\begin{remark}
A stronger version of the affine equivariance can be shown to hold. Let $A = ( A_{ij} )_{i=1}^k {}_{j=1}^p$ where $A_{ij}$ is a linear operator from $\mathcal{H}_j$ to some suitable Hilbert space $\ca G \lo i$, and $k$ is any positive integer. Then we still have $\Sigma(AX) = A \Sigma(X) A^*$, a property that is in $\mathbb{R}^p$ called full affine equivariance.
\end{remark}
Part i of Lemma \ref{lem:S_properties} guarantees the existence of the spectral decomposition of $\Sigma_{X X}$ into a sum of rank-1 operators:
\begin{equation}
\Sigma \lo {XX} = \sum_{k=1}^\infty \lambda_k ( \phi_k \otimes \phi_k ),
\end{equation}
where $(\phi_k, \lambda_k)$ are eigenvector-eigenvalue pairs, $\{ \phi_k \}_{k=1}^\infty$ is an orthonormal basis of $\mathcal{H}$ and the eigenvalues satisfy $\lambda_1 \geq \lambda_2 \geq \cdots \geq 0$. This representation will be used next to define the independent component model in $\mathcal{H}$.

\section{Independent component analysis in $\mathcal{H}$}\label{sec:independent}

\subsection{Independent component model in $\mathcal{H}$}

We say that $X \in \mathcal{X}^2 (\mathcal{H})$ follows the $\mathcal{H}$-valued independent component model if there exists a matrix of operators $\Gamma = ( \Gamma_{ij} )_{i,j=1}^p \in \mathcal{L}(\mathcal{H})$ such that
\begin{equation}\label{eq:icm_function_1}
\Gamma X =  Z,
\end{equation}
where $Z = (Z_1, \ldots, Z_p)$ is a random element in $\mathcal{H}$ having mutually independent component functions. We define two random elements $X: \Omega \rightarrow \mathcal{G}_1$ and $Y: \Omega \rightarrow \mathcal{G}_2$, not necessarily having values in the same space, to be independent if $E \{ \textbf{q}_1(X) \textbf{q}_2(Y)^T \} = E \{ \textbf{q}_1(X) \} E \{ \textbf{q}_2(Y)^T \}$ for all $\textbf{q} \lo 1(X) \in \ca X \hi 2 ( \mathbb{R}^{p_1})$, $\textbf{q} \lo 2 (Y) \in \ca X \hi 2 (\mathbb{R}^{p_2})$ with $p_1, p_2 \in \mathbb{N}$. The objective in the $\mathcal{H}$-valued independent component analysis is to estimate some unmixing operator $\Gamma$ such that $\Gamma X$ has independent component functions.

Like its vector-valued analogy in \eqref{eq:icm_vector}, the operator $\Gamma$ in model (\ref{eq:icm_function_1}) is not uniquely defined. If one applies to both sides of \eqref{eq:icm_function_1} any diagonal operator $D \in \mathcal{L}(\mathcal{H})$, the right-hand side still retains independent component functions. This implies that, without further assumptions, we cannot hope to find any unique functional form for the component functions. Indeed, as we show later in this section, our proposed methods actually estimate $\{B_j Z_j\}_{j=1}^p$ where $B_j \in \mathcal{L}(\mathcal{H}_j)$, $j=1, \ldots , p$. However, this identifiability issue does not affect our goal of discovering independent components, as the resulting vector of functions has independent component functions regardless of the form of $D$.

We will next approach the problem by extending two methods of vector-valued independent component analysis, FOBI and JADE, to the case of vector-valued random functions.

\subsection{Standardization of a random vector-valued function}

The first step in vector-valued independent component analysis is the standardization of $\textbf{x}$ by the inverse square root of the covariance matrix $\mathrm{cov}(\textbf{x})$. However, like in \cite{li2015functional}, the fact that the inverses of compact operators are unbounded means that we must resort to additional assumptions. Let $\{ \phi_k \}_{k=1}^\infty$ be the orthonormal basis of $\mathcal{H}$ consisting of the eigenvectors of $\Sigma \lo {XX}$ in decreasing order according to the corresponding eigenvalues. For a fixed $d \in \mathbb{N}$, let $\mathcal{M}_d = \mbox{span}(\{ \phi_k \}_{k=1}^d)$ be the subspace of $\mathcal{H}$ spanned by the $d$ first eigenvectors of $\Sigma \lo {XX}$. The simplifying assumption we make is the following.

\begin{assumption}\label{assu:inf_to_fin}
The component functions of $X$ are dependent only along the $d$ orthogonal directions $\{ \phi_k \}_{k=1}^d$. That is, if $R \lo d =  \tsum \lo {k=d+1} \hi \infty \langle X, \phi \lo k\rangle \lo {\ca H} \phi_k$, then the $p$ components of $R \lo d$ are independent.
\end{assumption}

In vector-valued independent component analysis this assumption is naturally always satisfied by picking simply $d = p$. One interpretation for the assumption in the current case is that the majority of the structure of the independent component functions is noise, meaning that the signal in the function $Z$ is in some sense finite-dimensional.

%To gain more insight into this assumption, note that if $\langle X, \phi \lo k \rangle \lo {\ca H}$ are Gaussian for $k = d + 1, d + 2, \ldots$, then the $p$ components of $R \lo d$ are independent. This is because the $\langle X, \phi \lo {i} \rangle \lo {\ca H}$'s are by definition uncorrelated. Thus a sufficient condition for Assumption \ref{assu:inf_to_fin} is that $X$ only has a  finite number of non-Gaussian components. Given that Gaussian signals are generally regarded as uninteresting in signal processing, the above assumption is satisfied if  there are finitely many interesting signals in $X$.

Note that $\{ \phi \lo k \}$ may not span the entire $\ca H$. However, by definition they are guaranteed to span $\cran (\Sigma \lo {XX})$, the closure of the range space of $\Sigma \lo {XX}$. Because $\Sigma \lo {XX}$ is self-adjoint,  $\cran(\Sigma \lo {XX} ) \oc = \ker ( \Sigma \lo {XX} )$, the kernel space of $\Sigma \lo {XX}$. Meanwhile, for any $f \in \ker ( \Sigma \lo {XX} )$, we have
\begin{align*}
\langle f, \Sigma \lo {XX} f \rangle \lo {\ca H} = \var \left\{ \langle f, X \rangle \lo {\ca H} \right\} = 0,
\end{align*}
which implies that $\langle f, X \rangle \lo {\ca H} = \mbox{constant}$ almost surely. Since this holds for the special case $X (\omega) = 0$, we have $\langle f, X \rangle \lo {\ca H} = 0$ almost surely. This means $f$ is orthogonal to the support of $X$. Since such  functions are of no interest to us, we can, without loss of generality, reset $\ca H$ to be $\cran (\Sigma \lo {XX})$, as we will do for the rest of the paper.

%HERE 10/21/2017

For an arbitrary subspace $\mathcal{M} \subset \mathcal{H}$, let $P_\mathcal{M}$ and $Q_\mathcal{M}$ denote the orthogonal projections on to $\mathcal{M}$ and $\mathcal{M}^\perp$, respectively. Then Assumption \ref{assu:inf_to_fin}    says that we can without loss of generality consider the projections $X^{(d)} = P_{\mathcal{M}_d} X$ instead of the original observations $X$. This simplifies the model \eqref{eq:icm_function_1} to the form
\begin{equation}\label{eq:icm_function_2}
\Gamma_0 X^{(d)} =  Z^{(d)},
\end{equation}
where $X^{(d)}, Z^{(d)}$ are random functions in $\mathcal{M}_d$, the component functions of $Z^{(d)}$ are independent and $\Gamma_0 \in \mathcal{L}(\mathcal{M}_d)$ is assumed to be invertible.

\begin{remark}
Later in this section the proposed methods are shown to be Fisher consistent, meaning that under the model \eqref{eq:icm_function_1} and Assumption \ref{assu:inf_to_fin} the final independent component scores are invariant to injective transformations $X \mapsto (P_{\mathcal{M}_d} A P_{\mathcal{M}_d} + Q_{\mathcal{M}_d}) X$, where $A \in \mathcal{L}(\mathcal{M}_d)$. However, as our estimation methods crucially depend on the existence of a random function $Z$ it is not meaningful to speak of affine equivariance outside the model in the same general sense that holds for both vector-valued FOBI and JADE, see \cite{miettinen2015fourth}.
\end{remark}

With this, we are now ready to present the first step towards the estimation of $Z$, an analogy for Lemma 3 in \cite{li2015functional}. In the following, for an $A \in \ca L ( \ca M \lo d)$, let $A^{- 1/2}$ denote the self-adjoint inverse square root of the self-adjoint linear operator $A$ within $\ca M \lo d$, that is, $A^{- 1/2} A A^{- 1/2} = P \lo {\ca M \lo d}$.

\begin{lemma} \label{lem:standardization}
Assume that $X\in \mathcal{X}^2(\mathcal{H})$ follows the model \eqref{eq:icm_function_2}. Then
\begin{equation*}
\Sigma(X^{(d)})^{- 1/2} X^{(d)} = U_0 \Sigma(Z^{(d)})^{- 1/2} Z^{(d)},
\end{equation*}
for some unitary operator $U_0 \in \mathcal{L}(\mathcal{M}_d)$.
\end{lemma}

The standardized functions are in the following denoted by $\tilde{X} = \Sigma(X^{(d)})^{-1/2} X^{(d)}$ and $\tilde{Z} = \Sigma(Z^{(d)})^{-1/2} Z^{(d)}$ and naturally satisfy $\Sigma ( \tilde{X} ) = \Sigma ( \tilde{Z} ) =  P \lo {\ca M \lo d}$. The next step towards finding $Z$ is the estimation of the unknown unitary operator $U_0$ in Lemma \ref{lem:standardization}. As described in the introduction both FOBI and JADE approach it via matrices of fourth cross-cumulants and before continuing we first define operatorial counterparts for them.

\subsection{The fourth cross-cumulant operators $C^{ij}(X)$}

In this section we assume that the zero-mean random function $X \in \mathcal{X}^4 (\mathcal{M}_d)$ resides in the $d$-dimensional space $\mathcal{M}_d$ spanned by the fixed orthonormal basis $\{ \phi_k \}_{k=1}^d$. We define the $(i, j)$th fourth cross-cumulant of $X$ with respect to the basis $\{ \phi_k \}_{k=1}^d$ to be
\begin{align*}
C^{ij}(X) = &E \left\{ \langle X, \phi_i \rangle_\mathcal{H} \langle X, \phi_j \rangle_\mathcal{H} ( X \otimes X ) \right\} - E \left\{ \langle X', \phi_i \rangle_\mathcal{H} \langle X', \phi_j \rangle_\mathcal{H} ( X \otimes X ) \right\} \numberthis \label{eq:cumulant_operator_1} \\
- &E \left\{ \langle X', \phi_i \rangle_\mathcal{H} \langle X, \phi_j \rangle_\mathcal{H} ( X' \otimes X ) \right\} - E \left\{ \langle X', \phi_i \rangle_\mathcal{H} \langle X, \phi_j \rangle_\mathcal{H} ( X \otimes X' ) \right\},
\end{align*}
where $i,j = 1, \ldots , d$ and the random function $X'$ is an independent copy of $X$. Repeated application of the Cauchy-Schwarz inequality shows that, for example, the first term in \eqref{eq:cumulant_operator_1} satisfies
\[
E \left\{ \| \langle X, \phi_i \rangle_\mathcal{H} \langle X, \phi_j \rangle_\mathcal{H} ( X \otimes X ) \|\lo {\op} \right\} \leq E \left( \| X \|^4_\mathcal{H} \right) < \infty,
\]
implying that the first term of \eqref{eq:cumulant_operator_1} exists as a uniquely defined bounded linear operator in $\mathcal{M}_d$. Similar considerations for the other terms show that the operator $C^{ij}(X) \in \mathcal{L}(\mathcal{M}_d)$ is then well-defined. Our main interest is in standardized random functions, $\Sigma(X) = P \lo {\ca M \lo d}$, and the following lemma provides a simplified form for \eqref{eq:cumulant_operator_1} in that case.
\begin{lemma}\label{lem:cumulant_operator_2}
Let the zero-mean random function $X \in \mathcal{X}^4 (\mathcal{M}_d)$ satisfy $\Sigma(X) = P \lo {\ca M \lo d} $. Then we have
\[C^{ij}(X) = E \left\{ \langle X, \phi_i \rangle_\mathcal{H} \langle X, \phi_j \rangle_\mathcal{H} ( X \otimes X ) \right\} - \delta_{ij} P \lo {\ca M \lo d} - \phi_i \otimes \phi_j - \phi_j \otimes \phi_i.\]
\end{lemma}

The operator $C^{ij}$ in Lemma \ref{lem:cumulant_operator_2} closely resembles the cross-cumulant matrix \eqref{eq:cumulant_matrix} for standardized random vectors and is next shown to serve similar purposes in constructing our versions of FOBI and JADE in $\mathcal{H}$.

\begin{theorem}\label{theo:Cij_properties}
Assume that $Z \in \mathcal{X}^4(\mathcal{M}_d)$ has independent component functions and that $\Sigma(Z) = P \lo {\ca M \lo d}$. Then we have for any unitary matrix of operators $U = ( U_{kl} )_{k,l=1}^p \in \mathcal{L}(\mathcal{M}_d)$ and for any $i,j = 1, \ldots , d$:
\[ C^{ij}(U Z) = U D^{ij} U^*,\]
where $D^{ij} = D^{ij}(U,Z)$ is a diagonal matrix of operators with the diagonal operators
\begin{align*}
D^{ij}_{kk} = E \left\{ (Z_k \otimes Z_k) (\xi_{ik} \otimes \xi_{jk}) (Z_k \otimes Z_k) \right\} - \langle \xi_{ik}, \xi_{jk} \rangle_k P_k - (\xi_{ik} \otimes \xi_{jk}) - (\xi_{jk} \otimes \xi_{ik}), \numberthis \label{eq:Cij_diagonal}
\end{align*}
for $k = 1, \ldots , p$, where $\xi_i = (\xi_{i1}, \ldots , \xi_{ip}) = U^* \phi_i$ and $P_k$ is the projection operator from the $k$th component space of $\mathcal{H}$ to the $k$th component space of $\mathcal{M}_d$.
\end{theorem}

Theorem \ref{theo:Cij_properties} essentially says that $U^*$ diagonalizes (as in a diagonal operator) the operator $C^{ij}(U Z)$ for every choice of $i,j=1, \ldots ,d$ and these decompositions provide us a mean of finding the missing unitary operator $U$. Our version of JADE will later utilize all $p^2$ of these operators and for FOBI we use just a subset of them, captured by the FOBI-operator $C(X) \in \mathcal{L}(\mathcal{M}_d)$,
\begin{equation}\label{eq:C_definition}
C(X) = \sum_{i=1}^d C^{ii}(X).
\end{equation}
The next theorem gives some useful properties of this operator.

\begin{theorem}\label{theo:C_properties}
Assume that $Z \in \mathcal{X}^4(\mathcal{M}_d)$ has independent component functions and that $\Sigma(Z) = P \lo {\ca M \lo d}$. Then, for any unitary matrix of operators $U = ( U_{kl} )_{k,l=1}^p \in \mathcal{L}(\mathcal{M}_d)$, the FOBI-operator \eqref{eq:C_definition} satisfies
\[C(UZ) = U C(Z) U^* = U  D  U^*, \]
where $D = C(Z)$ is a diagonal matrix of operators with the diagonal entries
\[
D_{kk} = E \left\{ (Z_k \otimes Z_k)^2 \right\} - (d_k + 2) P_k,
\]
where $d_k = \mbox{dim}[ \mbox{span}(\{ \phi_{mk} \}_{m=1}^d)]$, the dimension of the $k$th component space, and $P_k$ is the projection operator from the $k$th component space of $\mathcal{H}$ to the $k$th component space of $\mathcal{M}_d$.
\end{theorem}

The first equality in Theorem \ref{theo:C_properties} does not need the independence of the component functions of $Z$ but actually holds for all standardized $Z \in \mathcal{X}^4(\mathcal{M}_d)$, as long as the operator $U$ is unitary. This property of the functional $B: \mathcal{X}^4(\mathcal{M}_d) \rightarrow \mathcal{L}(\mathcal{M}_d)$ is called unitary equivariance.

Recall from the introduction that in FOBI we diagonalize a single matrix and in JADE multiple matrices simultaneously. The functional analogy for the former is the spectral decomposition of $C(X)$ and for the latter we define next the joint diagonalization of a set of operators. Namely, define the joint diagonalizer of a finite set of operators, $\mathcal{S} = \{ S_i \mid S_i \in \mathcal{L}(\mathcal{M}_d), i=1, \ldots, I \}$, to be the orthonormal basis $\{ \psi_k \}_{k=1}^d$ of $\mathcal{M}_d$ that maximizes the objective function
\begin{equation}\label{eq:joint_diag}
w \left( \psi_1, \ldots , \psi_d \right) = \sum_{i=1}^I \sum_{k=1}^d \langle \psi_k, S_i \psi_k \rangle_{\mathcal{H}}^2.
\end{equation}

In the previous paragraphs we have discussed two kinds of diagonality, the diagonality in the sense of diagonal operators in Theorems \ref{theo:Cij_properties} and \ref{theo:C_properties} and the diagonality in the sense of the spectral decomposition. The final tool we need for the estimation of the independent functions is a connection between these two concepts. Recall that by a canonical vector we mean any element of $\mathcal{H}$ which has at most one non-zero component. The needed connection is now provided by the next pair of lemmas which show that (under suitable assumptions) the spectral decomposition and joint diagonalization of diagonal operators mimic the eigendecomposition and joint diagonalization of diagonal real matrices in the sense that the spectral decompositions and the joint diagonalizer of a set of diagonal operators consist entirely of canonical vectors.

\begin{lemma}\label{lem:diagonal_spectra}
Let $D \in \mathcal{L}(\mathcal{H})$ be a diagonal matrix of operators with finite rank $d$ and let its spectral decomposition be
\begin{equation*}
D = \sum_{k=1}^d \tau_k (\psi_k \otimes \psi_k)
\end{equation*}
where the eigenvalues $\{ \tau_k \}_{k=1}^d$ are distinct. Then the eigenvectors $\{ \psi_k \}_{k=1}^d$ are canonical.
\end{lemma}

\begin{lemma}\label{lem:joint_diag}
Let $\mathcal{S} = \{ S_i \}_{i=1}^I$ be a finite collection of bounded linear operators in $\mathcal{M}_d$ and let $\{ \psi_k \}_{k=1}^d$ be an orthonormal basis of $\mathcal{M}_d$. Then we have
\begin{align*}
w \left( \psi_1, \ldots , \psi_d \right) \leq \sum_{i=1}^I \| S_i \|_{HS}^2,
\end{align*}
where $\| \cdot \|_{HS}$ is the Hilbert-Schmidt norm and an equality is reached if and only if each $\psi_k$ is an eigenvector of each $S_i$, $k = 1, \ldots , d$, $i = 1, \ldots , I$. In particular, if all operators in $\mathcal{S}$ are diagonal and share an eigenbasis then the elements of the joint diagonalizer are canonical.
\end{lemma}

\subsection{Finding the unitary transformation $U_0$}

Using the previously defined fourth cross-cumulant operators we next formulate the functional counterparts for the steps taken in vector-valued FOBI and JADE to estimate the orthogonal matrix $\textbf{U}$. %As described in the introductory section, FOBI uses a single linear combination of the fourth cross-cumulant operators while JADE considers them all simultaneously.

\begin{definition}
Let $X \in \mathcal{X}^4(\mathcal{H})$ follow the model \eqref{eq:icm_function_2}. Then we define
\begin{itemize}
\item[i)] FOBI-basis of $X$ is the set $\{ \psi^F_k \}_{k=1}^d$ of eigenfunctions of the FOBI-operator $C(\tilde{X})$,
\item[ii)] JADE-basis of $X$ is the joint diagonalizer $\{ \psi^J_k \}_{k=1}^d$ of the set of operators $\mathcal{C} = \{ C^{ij}(\tilde{X}) \}_{i,j=1}^d$.
\end{itemize}
\end{definition}

In the next theorem Lemmas \ref{lem:diagonal_spectra} and \ref{lem:joint_diag} are applied respectively to the FOBI-basis and JADE-basis to find $U_0$. However, to guarantee consistency we need to make some additional assumptions which guarantee that the eigenbases are unique up to signs and order. For the FOBI-solution we need the following.
\begin{assumption}\label{assu:fobi_assumption}
The eigenvalues of $C(\tilde{Z})$ are distinct.
\end{assumption}
One consequence of Assumption \ref{assu:fobi_assumption} is that FOBI cannot estimate two latent functions having the same distribution. For JADE the corresponding assumption is much more relaxed but to use Lemma \ref{lem:joint_diag} we first need the additional assumption that all the diagonal operators in Theorem \ref{theo:Cij_properties} share a common eigenbasis.
\begin{assumption}\label{assu:shared_eigenbasis}
The operators $D^{ij}(U_0, \tilde{Z})$, $i,j = 1, \ldots p$, have a common eigenbasis.
\end{assumption}
While this sounds somewhat stringent, in Section \ref{sec:sample} discussing the sample version of the method we show that Assumption \ref{assu:shared_eigenbasis} is in fact not that strict, and is satisfied under some general conditions and choices of $d$. The need for the next assumption guaranteeing the uniqueness of the eigenbasis for JADE now follows directly from the equality condition in Lemma \ref{lem:joint_diag}.
\begin{assumption}\label{assu:jade_assumption}
For each pair $(\psi^J_k, \psi^J_l)$, $k, l = 1, \ldots ,d$, there exists a pair $(i, j)$, $i,j = 1, \ldots d$, such that the eigenvalues of $D^{ij}(U_0, \tilde{Z})$ related to $\psi^J_k$ and $\psi^J_l$ are distinct.
\end{assumption}

The next theorem finally proves the Fisher consistency of our approach by showing how the FOBI-basis and JADE-basis can be used to estimate the independent component functions.

\begin{theorem}\label{theo:unmixing}
Let $X \in \mathcal{X}^4(\mathcal{H})$ follow the model \eqref{eq:icm_function_2} and let $\{ \psi^F_k \}_{k=1}^d$ and $\{ \psi^J_k \}_{k=1}^d$ be the FOBI-basis and JADE-basis of $X$, respectively. Assume further that either Assumption \ref{assu:fobi_assumption} (FOBI) or Assumptions \ref{assu:shared_eigenbasis} and \ref{assu:jade_assumption} (JADE) are satisfied. Then the FOBI and JADE estimators of the latent functions are respectively the $d$ elements of $\mathcal{X}^4(\mathcal{H})$ given as
\begin{equation*}
\hat{Z}^F_k = (\psi^F_k \otimes \psi^F_k) \tilde{X}, \quad k=1, \ldots , d, \quad \mbox{and} \quad \hat{Z}^J_k = (\psi^J_k \otimes \psi^J_k) \tilde{X}, \quad k=1, \ldots , d,
\end{equation*}
where each estimator $\hat{Z}^\bullet_k$ corresponds to exactly one latent function $Z_j$.
\end{theorem}

The proof of Theorem \ref{theo:unmixing} shows that for each $k = 1, \ldots , d$ the procedure actually recovers the $p$-variate function $\hat{Z}_k^\bullet =  \langle \psi^\bullet_{k}, \tilde{X} \rangle_\mathcal{H} \psi^\bullet_{k} = \langle h^\bullet_{k}, \tilde{Z} \rangle_\mathcal{H} \psi^\bullet_{k}$ where the only dependency on the latent function $Z$ is through the inner product $\langle \psi^\bullet_{k}, \tilde{X} \rangle_\mathcal{H} = \langle h^\bullet_{k}, \tilde{Z} \rangle_\mathcal{H}$. Furthermore, every $h_k^\bullet$ is canonical, meaning that each of the estimates $\hat{Z}^\bullet_k$ contains information on exactly one latent component $Z_j$ and this information is entirely contained in the single inner product, $\langle \psi^\bullet_{k}, \tilde{X} \rangle_\mathcal{H}$. In the following we will refer to these inner products as the independent component scores. As more than one score can be related to a single latent function $Z_j$, the $d$-vector of independent component scores can further be divided into $m$ mutually independent subvectors, $\textbf{Z}_{(l)} \in \mathbb{R}^{d_l}$, $\sum_{l=1}^m d_l = d$, so that each subvector corresponds to a single latent function $Z_j$.

\section{The methods in practice}\label{sec:sample}

\subsection{Sample versions of the methods}

For deriving the sample version of the proposed method we make the simplifying assumption that the component spaces are the same, $\mathcal{H}_1 = \cdots = \mathcal{H}_p$. The generalization to the case of different component spaces follows easily.

Let $X \hii 1, \ldots, X \hii n$ be a random sample of $X$. Here, we use   superscript to represent the position in a sample, to differentiate from the subscript in $X \lo i$ which represents the $i$th component of $X$. Furthermore, let  $X \lo {ij}$ represent the $j$th component of $X \hii i$. Although our theory is based on infinite-dimensional spaces, our observations are always finite-dimensional and so let $X \lo {ij}(t_{m,ij})$ denote the value of the $j$th component function of the $i$th observation at the time point
\begin{align*}
\{ t_{m,ij}: m = 1, \ldots, M \lo {ij}, j= 1,\ldots, p, i = 1, \ldots, n \}  .
\end{align*}
We thus allow the measurement times and the numbers of measurements to differ across both observations and components. The underlying assumption in functional data analysis is that the observed values $X \lo {ij}(t_{m,ij})$ correspond to latent (smooth) functions that we observe only at the discrete times $t \lo {m,ij}$. The first step in implementing the method is thus to express all the observations as functions using some suitable basis.

For approximating the space $\mathcal{H}$, fix a $K$-element basis $\mathcal{G}_0 = \{ g_k \}_{k=1}^K$, the span of which we denote as $\mathcal{M}_0$. The functional approximations $\hat{x}_{ij}(t) = \sum_{k=1}^K \hat{c}_{ijk} g_k(t)$ of the observed curves in $\mathcal{M}_0$ can be found as
\begin{equation*}
(\hat{c}_{ij1}, \ldots , \hat{c}_{ijK})^T = \mbox{argmin} \sum_{m = 1}^{M_{ij}} \left\{ X \lo {ij}(t \lo {m,ij}) - \sum_{k=1}^K c_{ijk} g_k(t \lo {m,ij}) \right\}^2,
\end{equation*}
which is a least-squares type problem. Having estimated the coordinates $\hat{c}_{ijk}$ we denote in the following the coordinate vector of the $j$th component function of the $i$th observation in the basis $\mathcal{G}_0$ as $[X \lo {ij}]_{\mathcal{G}_0} = (c_{ij1}, \ldots , c_{ijK})^T \in \mathbb{R}^K$. Consider then the $pK$-dimensional product space $\mathcal{M} = \mathcal{M}_0 \times \cdots \times \mathcal{M}_0$. The space $\mathcal{M}$ then has the natural direct sum basis $\mathcal{G} = \mathcal{G}_0 \oplus \cdots \oplus \mathcal{G}_0$.  The stacked vector of the coordinates of all $p$ component functions of the $i$th observation in the basis $\mathcal{G}$ is denoted by $[X \hii i]_\mathcal{G} = ([X \lo {i1}]_{\mathcal{G}_0}^T, \ldots ,[X \lo {ip}]_{\mathcal{G}_0}^T)^T \in \mathbb{R}^{pK}$ and the matrix of all coordinates of all observations by $[X]_\mathcal{G} = ([X \hii 1 ]_\mathcal{G}, \ldots , [X \hii n]_\mathcal{G})^T \in \mathbb{R}^{n \times pK}$. We assume without loss of generality that the coordinate representations of the observations are centered, $\sum_{i=1}^n [X \lo {ij}]_{\mathcal{G}_0} = \textbf{0}$, $j = 1, \ldots , p$.

Let $\textbf{G}_\mathcal{G} = ( \langle g_k , g_{k'} \rangle \lo {\ca H} )_{k,k'=1}^K$ denote the Gram matrix of a basis $\mathcal{G} = \{ g_k \}_{k=1}^K$. For orthonormal bases the Gram matrix equals the identity matrix and if $\mathcal{G}$ is a direct sum basis $\mathcal{G} = \mathcal{G}_0 \oplus \cdots \oplus \mathcal{G}_0$ then clearly $\textbf{G}_\mathcal{G} = \mbox{diag}(\textbf{G}_{\mathcal{G}_0}, \ldots , \textbf{G}_{\mathcal{G}_0}) = (\textbf{I}_p \otimes \textbf{G}_{\mathcal{G}_0})$, where $\textbf{G}_{\mathcal{G}_0}$ is the Gram matrix of the basis $\mathcal{G}_0$ and $\otimes$ is the Kronecker product between matrices. The next theorem now describes how the coordinate representations can be used to carry out the proposed methods in practice.

\begin{theorem}\label{theo:coordinates}

Let $[\hat{\Phi}]\lo {\ca G} \in \mathbb{R}^{pK \times d}$ contain the $d$ first eigenvectors of the matrix $(1/n) [X] \lo {\ca G}^T [X] \lo {\ca G} (\textbf{I} \lo p \otimes \textbf{G} \lo {\ca G _0})$ and let the diagonal matrix $\boldsymbol{\Lambda}_d \in \mathbb{R}^{d \times d}$ hold the corresponding eigenvalues as its diagonal elements. Then, let $[\tilde X \hi i]\lo {\ca V} = \boldsymbol{\Lambda}_d^{-1/2} [\hat{\Phi}]^T \lo {\ca G} (\textbf{I} \lo p \otimes \textbf{G} \lo {\ca G _0}) [X \hi i] \lo {\ca G}  \in \mathbb{R}^d$, $i=1, \ldots , n$, contain the coordinates of the standardized observations in the eigenbasis. Finally, let

\begin{itemize}
\item[i)] the columns of $[\hat{\Psi}^F]\lo {\ca V} \in \mathbb{R}^{d \times d}$ be the eigenvectors of the matrix
\begin{equation*}
\frac{1}{n} \sum_{i=1}^n [\tilde X \hi i]\lo {\ca V}^T [\tilde X \hi i]\lo {\ca V} \cdot  [\tilde X \hi i]\lo {\ca V} [\tilde X \hi i]\lo {\ca V}^T  - (d + 2) \textbf{I}_d,
\end{equation*}
\item[ii)] the columns of $[\hat{\Psi}^J]\lo {\ca V}  = ([\hat{\psi}^J_1]\lo {\ca V}, \ldots , [\hat{\psi}^J_d]\lo {\ca V}) \in \mathbb{R}^{d \times d}$ be the orthonormal set of vectors satisfying
\begin{equation*}
[\hat{\Psi}^J]\lo {\ca V} = \underset{[\hat{\Psi}^J]\lo {\ca V}^T [\hat{\Psi}^J]\lo {\ca V} = \textbf{I}}{argmax}\sum_{k=1}^d \sum_{l=1}^d \sum_{m=1}^d \left\{ [\hat{\psi}^J_m]^T\lo {\ca V} ({}\lo {\ca V}[\hat{C}^{kl}(\tilde{X})]\lo {\ca V})  [\hat{\psi}^J_m]\lo {\ca V}  \right\}^2,
\end{equation*}
where ${}\lo {\ca V}[\hat{C}^{kl}(\tilde{X})]\lo {\ca V} = (1/n) \sum_{i=1}^n ([\tilde X \hi i]\lo {\ca V}^T \textbf{e}_k) ([\tilde X \hi i]\lo {\ca V}^T \textbf{e}_l) \cdot  [\tilde X \hi i]\lo {\ca V} [\tilde X \hi i]\lo {\ca V}^T  - \delta_{kl} \textbf{I}_d - \textbf{e}_k \textbf{e}_l^T - \textbf{e}_l \textbf{e}_k^T$.
\end{itemize}

Then, choosing either the FOBI-solution $[\hat{\Psi}^F]\lo {\ca V}$ or the JADE-solution $[\hat{\Psi}^J]\lo {\ca V}$ the independent component scores are given by
\begin{equation*}
\hat{\textbf{Z}}^i = [\hat{\Psi}^\bullet]^T\lo {\ca V} \boldsymbol{\Lambda}_d^{-1/2} [\hat{\Phi}]^T\lo {\ca G} (\textbf{I} \lo p \otimes \textbf{G} \lo {\ca G _0}) [X \hi i ]\lo {\ca G}.
\end{equation*}

\end{theorem}

The optimization problem required by the JADE-solution is easily solved with standard joint diagonalization techniques, e.g. the Jacobi angle algorithm, see \cite{cardoso1996jacobi}. An implementation of the algorithm can be found in the R-package JADE \citep{Rjade}. Theorem \ref{theo:coordinates} shows that the resulting vector of independent component scores is a linear transformation of the original vector of coordinates, $\hat{\textbf{Z}}^i = \textbf{A} [X \hi i ]\lo {\ca G}$ for some $d \times pK$ matrix $\textbf{A}$. Consequently, we can get interpretations for the independent component scores by considering the elements of $\textbf{A}$ and observing which of the original coordinates most influence each of the obtained scores. The same procedure is used in the standard principal component analysis where the elements of the matrix $\textbf{A}$ are called loadings. An example of such an interpretation will be given in the real data example in Section \ref{sec:examples}.

\subsection{Choosing the value of $d$}

We next give some rough guidelines on choosing an appropriate reduced dimension $d$. Naturally, we can estimate independent component scores corresponding to each latent function $Z_j$ only if $d \geq p$. Moreover, even if we put $d = p$ it could still happen that some of the component functions have too low variation and cannot fit amongst the $d$ eigenvectors of $\Sigma_{XX}$ with the highest eigenvalues. From this point of view it would thus make sense to increase $d$ further to make sure we capture all the latent functions. However, doing this also increases the odds of introducing more and more of the non-dependent part of the model (noise) to the estimation.

Further complication is brought in by Assumption \ref{assu:shared_eigenbasis} which in the sample version requires that all the diagonal matrices in the JADE-decomposition share a single eigenbasis. It can be shown that a sufficient condition for this is that each of the subvectors $\textbf{Z}_{(l)}$, $l = 1, \ldots , m$ has either length one or an elliptical distribution. This condition is more likely to be fulfilled for small values of $d$ and since $d = p$ is a natural meeting point for all these rules, allowing us to estimate all $p$ latent functions in the best case, we advocate the use of the value $d = p$ in practice. This rule of thumb will be used in the examples of the next section.

\section{Examples}\label{sec:examples}

\subsection{Simulation study}

In this simulation study we compare the two proposed methods to the alternative of applying only the principal component analysis part of the algorithm, that is, only projecting the data onto the space spanned by the first $d$ eigenfunctions of $\Sigma_{X X}$.

For our setting we used $p=4$ and considered for all four component functions the same $11$-element Fourier basis $\mathcal{G}_0$. The leading coefficients in the coordinate vectors of the component functions were generated either as $(u_1, g_1, \chi_1, e_1)$ (Setting 1) or as $(u_1, u_2, u_3, u_4)$ (Setting 2) where $u_1, u_2, u_3, u_4 \sim Uniform(0, 1)$, $g_1 \sim \Gamma(3, \sqrt{3})$, $\chi_1 \sim \chi^2_3$, $e_1 \sim Exp(1)$ and all the previous random variables were independent and standardized to have zero means and unit variances. The rest of the coordinates were independent standard normal. In the first setting all the ``signal'' components thus had distinct kurtoses and in the second setting they had identical kurtoses. We generated samples of sizes $n = 1000, 2000, 4000, 8000, 16000, 32000, 64000$ and mixed the individual generated functions, $Z_i = (Z_{i1}, Z_{i2}, Z_{i3}, Z_{i4})^T$, as $[Z^i]_{\mathcal{G}} \mapsto [X^i]_{\mathcal{G}} = \bo{\Omega} [Z^i]_{\mathcal{G}}$ with a random mixing matrix $\bo{\Omega} \in \mathbb{R}^{44 \times 44}$. For simplicity, we considered estimation only in the true case $d = 4$.

To obtain $\bo{\Omega}$ we first generated the matrix $\bo{\Omega}_0 = diag \left\{ (\textbf{B}_4)^{1/2}, \textbf{I}_{40} \right\}$, where $\textbf{B}_4 = \textbf{AA}^T + \lambda \textbf{I}_4$, the matrix $\textbf{A} \in \mathbb{R}^{4 \times 4}$ has independent standard normal elements and $\lambda = 0.5, 1.0, 1.5, 2.0, 2.5$ is a tuning parameter that controls how separated the spectra of the mixed and unmixed parts are. The mixing matrix $\bo{\Omega}$ is now obtained by permuting the rows and columns of $\bo{\Omega}_0$ so that only the leading coefficients of the component functions are going to be mixed in the transformations $[Z^i]_{\mathcal{G}} \mapsto \bo{\Omega} [Z^i]_{\mathcal{G}}$. This unorthodox procedure goes to ensure that the dependency between the four functions exists only in the directions given by the eigenvectors of $\Sigma_{X X}$ with the eigenvalues $a_j^2 + \lambda$, $j=1, \ldots , 4$, where $a_j$ are the singular values of $\textbf{A}$. Thus if $\lambda > 1$, the four largest eigenvalues always (on the population level) correspond to the directions of interest, meaning that  the assumptions of our model are fulfilled and we always pick the correct four eigenvectors. A similar mixing scheme was used also in \cite{li2015functional}.

Subjecting the data to our proposed independent component methods, both of them estimate a matrix
\[\textbf{W} = [\hat{\Psi}^\bullet]^T\lo {\ca V} \boldsymbol{\Lambda}_d^{-1/2} [\hat{\Phi}]^T\lo {\ca G} (\textbf{I} \lo p \otimes \textbf{G} \lo {\ca G _0}) \in \mathbb{R}^{d \times pK} = \mathbb{R}^{4 \times 44}, \]
see Section \ref{sec:sample}, while the principal component analysis uses only the matrix $\textbf{W} = [\hat{\Phi}]^T\lo {\ca G} (\textbf{I} \lo p \otimes \textbf{G} \lo {\ca G _0}) \in \mathbb{R}^{4 \times 44}$. The independent/principal component scores are then $ \textbf{W} [X \hi i]_{\mathcal{G}} = \textbf{W} \bo{\Omega} [Z \hi i]_{\mathcal{G}}$ and for the methods to successfully separate the independent component functions each row of the gain matrix $\textbf{W} \bo{\Omega}$ should pick from $[Z \hi i]  \lo {\mathcal{G}}$ coefficients relating only to a single component function. For assessing the performance of a single replication we first squared the elements of the estimated gain matrix and then summed row-wise over each block of size $4 \times 11$, resulting into a $4 \times 4$ matrix $\textbf{R}$. The closer the matrix $\textbf{R}$ is to the set $\mathcal{P}$ of matrices with a single non-zero element in each row and column, the better the result of estimation. To quantify this we use the minimum distance index \citep{ilmonen2010new}, $D(\textbf{R}) \in [0, 1]$, which has the value zero if and only if the separation is perfect, $\textbf{R} \in \mathcal{P}$.

\begin{figure}[t!]
\begin{center}
\includegraphics[width=\textwidth]{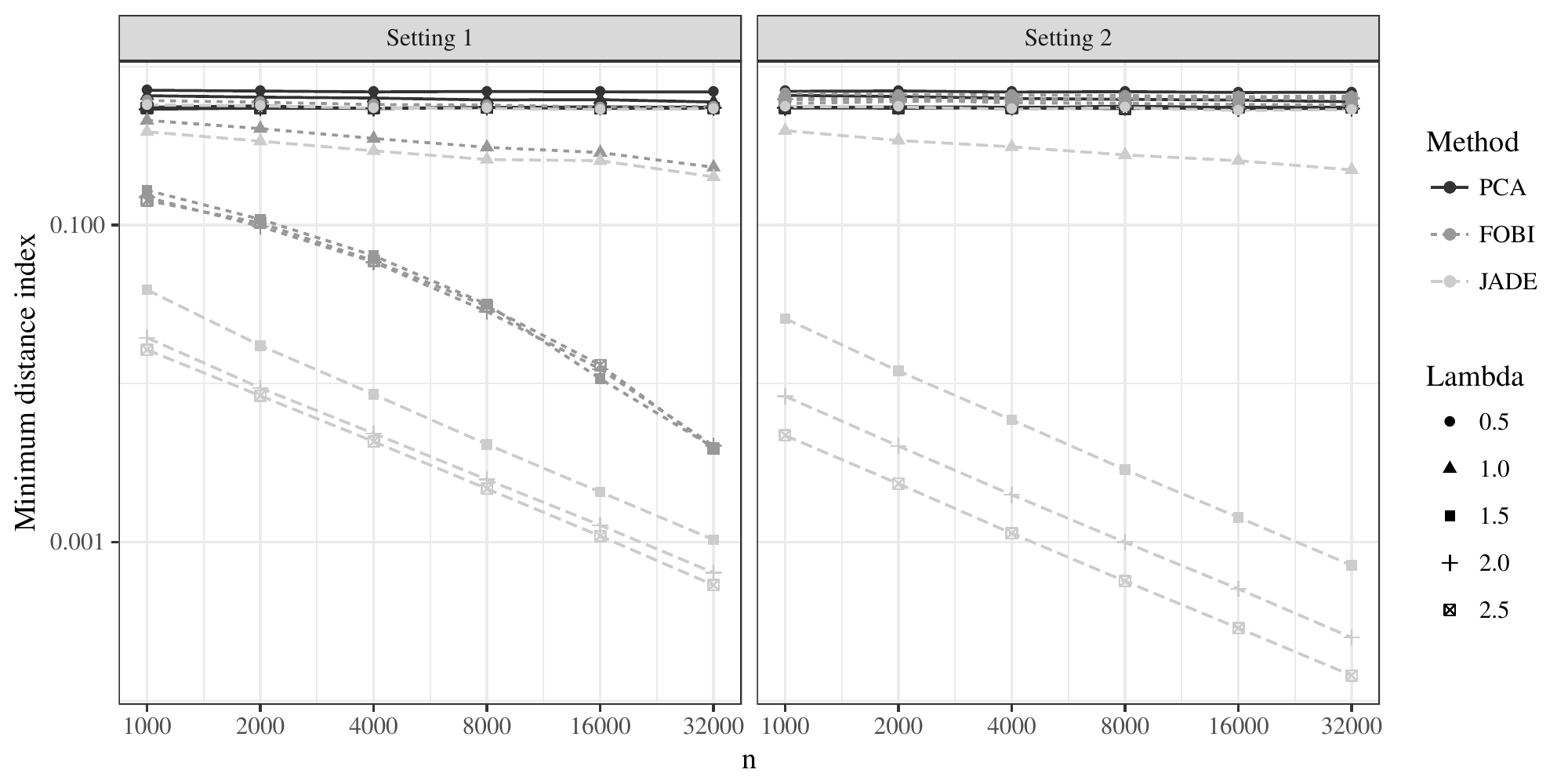}
\caption{The average minimum distance indices across different methods and settings in the simulation study. Lower value of the index indicates better separation. The scale of the $y$-axis is logarithmic.}
\label{fig:simu_1}
\end{center}
\end{figure}

From the results we expect that the principal component analysis fails to estimate the sources under all settings, as the orthogonal transformation found by it is not enough to undo our mixing by the general matrix $\textbf{B}_4^{1/2}$. The theory behind standard FOBI, on which our coordinate representation was seen to be based, says that FOBI cannot estimate components with matching kurtosis values \citep{cardoso1989source} as is the case with the identical uniform distributions in our Setting 2. On the other hand, both FOBI and JADE should be able to find the solution in Setting 1 with differing, non-zero kurtosis values, the latter most likely outmatching the former. The resulting mean minimum distance indices across 1000 replications for different settings and parameter values are shown in Fig. \ref{fig:simu_1} and distinctly verify our preconceptions. As discussed earlier, the separation fails on average if $\lambda \leq 1$ and we further see that the success of the separation is not particularly dependent on the value of $\lambda$, as long as we have $\lambda > 1$.

\subsection{Real data example}

%As balanced multimodal distributions are often characterized by low kurtosis values, kurtosis-based independent component analysis is often useful for revealing latent group information in multivariate data sets. To this end

We consider the \textit{uWave} gesture data set available from \texttt{http://zhen-wang.\\appspot.com/rice/projects\_uWave.html} \citep{liu2009uwave}.
At each day of the study the eight participants did ten repetitions of each of the eight gesture patterns in the Nokia gesture vocabulary \citep{kela2006accelerometer} using a Wii$^\circledR$ remote measuring the 3D-acceleration of the gesture. Each participant had a total of seven study days making the total number of observed samples 4480. Of these we discarded two samples which had a measurement only for a single time point. Of the observed $3$-variate curves ($x$, $y$ and $z$-acceleration) we further took the subset corresponding to the three visually most similar gestures, a square, a clockwise circle and a counterclockwise circle, making our data a sample of multivariate functional data with $n = 1679$ and $p = 3$. A standard Fourier basis of 11 functions was fitted to all observations of each component function.

%\begin{figure}[]
%\begin{center}
%\includegraphics[width=\textwidth]{fig_nokia_vocabulary.png}
%\caption{The eight gestures identified in Nokia gesture control study, \cite{kela2006accelerometer}, which is also the source of the figure.}
%\label{fig:nokia}
%\end{center}
%\end{figure}

\begin{figure}[t!]
\begin{center}
\includegraphics[width=\textwidth]{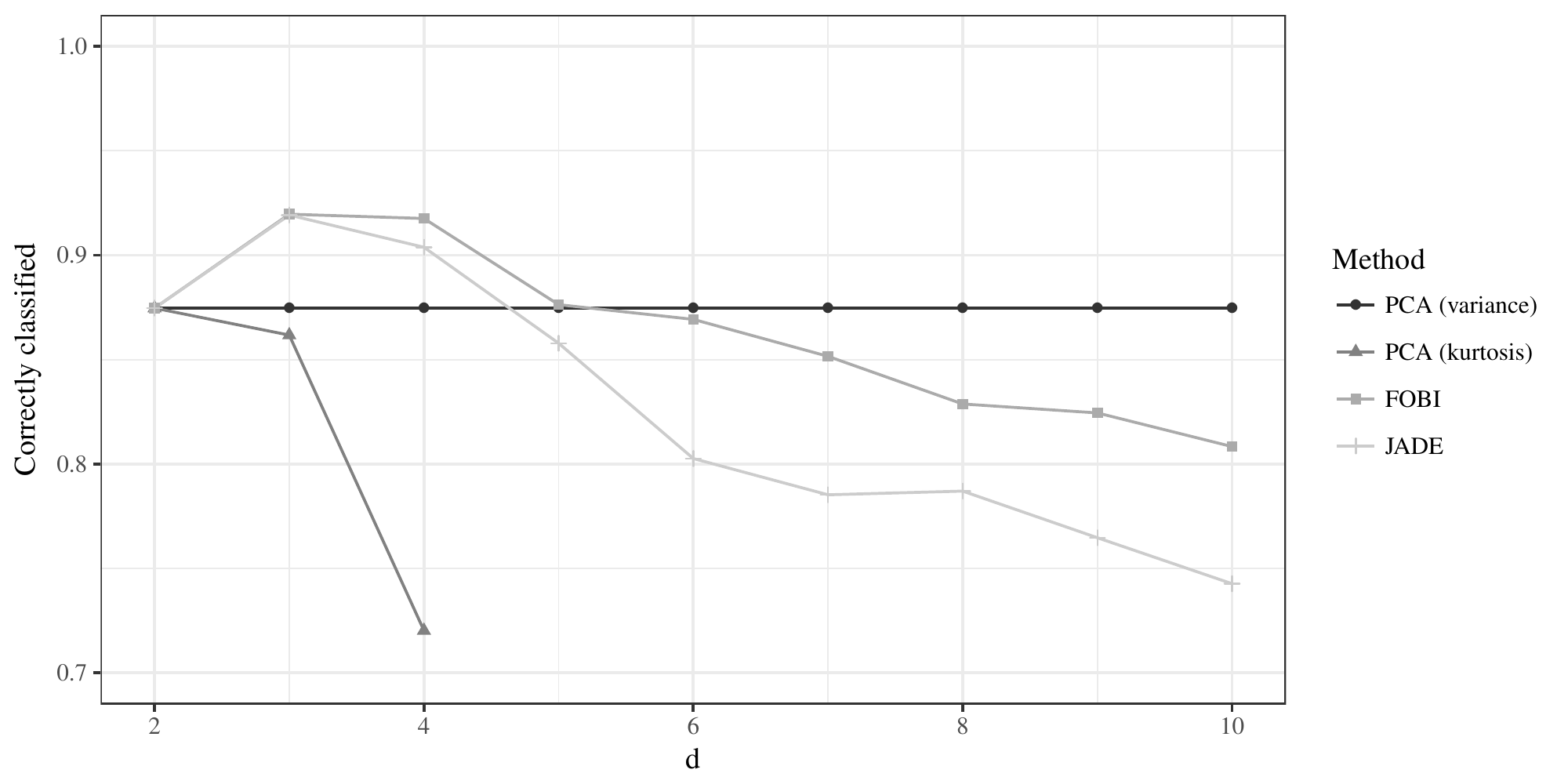}
\caption{The average proportions of correctly classified cases in the test set for different values of $d$.}
\label{fig:real_1}
\end{center}
\end{figure}

In pre-processing data, latent groups are most easily visually recognized from bivariate scatter plots and our objective is thus to extract from the data a pair of components that best reveal the latent group memberships. To evaluate the methods' capabilities for this we used the following scheme. For each of the 1000 replications we randomly partitioned the data into a training set of 400 observations and a test set of 1279 observations. Next, for each value of $d = 2, \ldots, 10$, the training set was subjected to either principal component analysis (conducted as in the previous example), FOBI or JADE. As low kurtosis is often an indicator of a multimodal distribution, for the independent component analysis methods we chose from the resulting independent component scores the two having the lowest fourth moments and for principal component analysis we considered two rules, taking the two scores with highest variances or taking the two scores with lowest fourth moments. Each chosen pair of scores was then used in quadratic discriminant analysis to create a classification rule and, finally, the proportion of correct classifications in the test set was computed for each rule.

The results are shown in Fig. \ref{fig:real_1} where the $y$-axis was cut from $0.7$ downwards to allow more accurate representation of the interesting part of the plot. The curve for principal component analysis using kurtosis as a criterion continued descending until hitting the $y$-value of around $0.5$ at $d=6$. The main points of interest include the following. All methods perform equally well when $d=2$ as then the chosen two components necessarily span the same space. Principal component analysis using variance as the criterion always chooses by definition the two first principal components regardless of the value of $d$, yielding a constant curve, and principal component analysis using kurtosis as the criterion clearly cannot find the relevant information at all. For $d=3,4$ FOBI and JADE are superior to principal component analysis in extracting the two components containing the classification information. Thus our heuristic suggestion of setting $d = p$ proved to be useful in this context.

\begin{figure}[t!]
\begin{center}
\includegraphics[width=\textwidth]{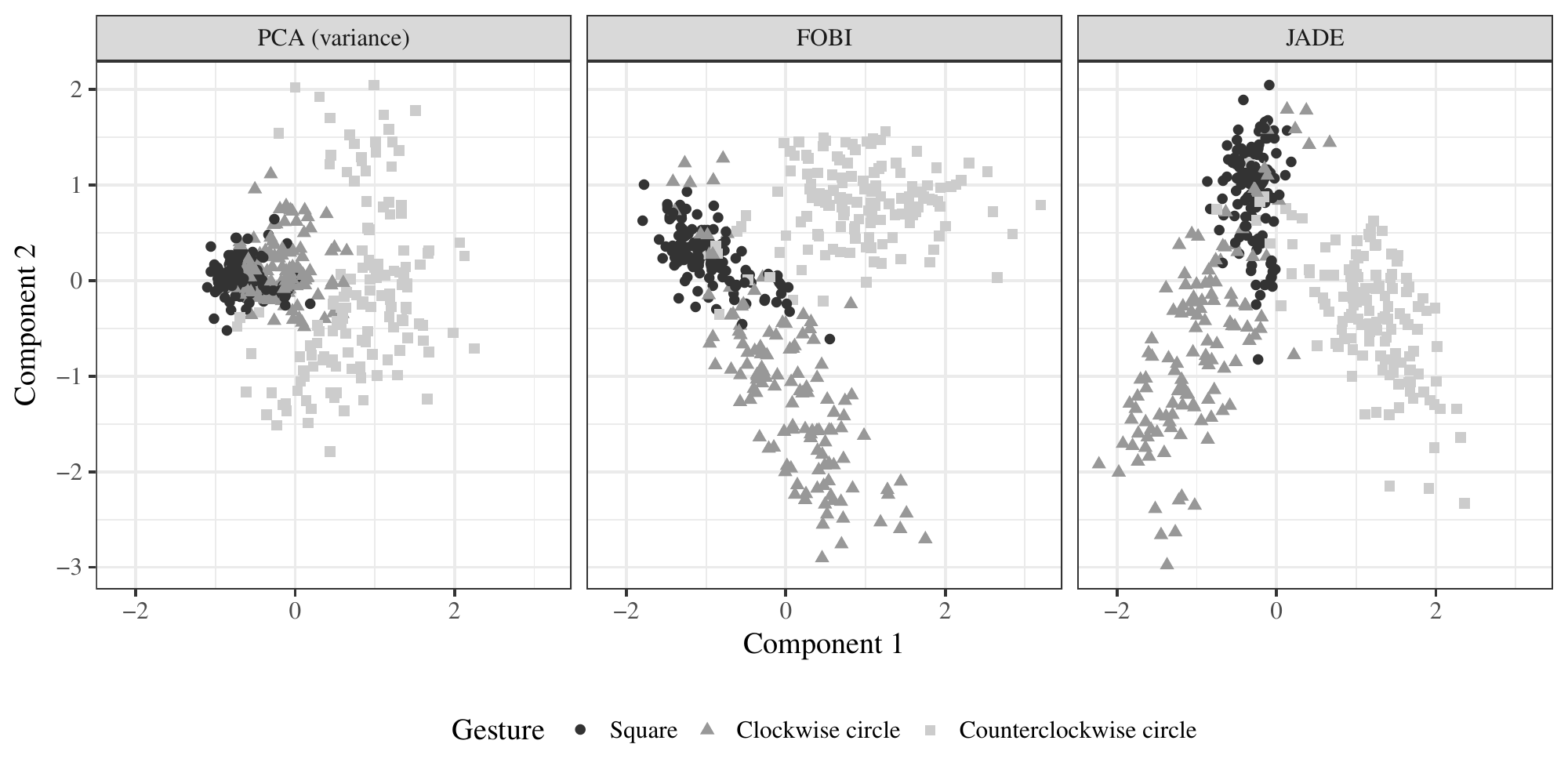}
\caption{Examples of the pairs of components found by the methods when $d=3$. }
\label{fig:real_2}
\end{center}
\end{figure}

Examples of the scatter plots of the pair of components extracted from the training data by the three methods for $d=3$ are given in Fig. \ref{fig:real_2} where the principal components have been scaled to better show the details. The figure shows that of the two components found by principal component analysis only the first one provides information on the separation of the group locations while for FOBI and JADE both components carry location information. Interpretations for the FOBI independent component scores can now be obtained by examining the loading matrix reproduced in Table \ref{tab:real_1} where any loadings with absolute value greater than 0.6 have been shaded. For example, the final element of the second row tells the contribution of the 11th basis vector of the second observed function $X_2$ to the first estimated score $\hat{Z}_1$. We can now make two main observations. First, no separation information is carried by the basis elements of order six or higher. Since the higher index functions in Fourier bases control the finer, high-frequency properties of the resulting functions this reveals that most of the classification information is expectedly contained in the large-scale properties of the movements and accelerations. Secondly, the $y$-acceleration hardly contributes to any of the scores, showing that only the $x$ and $z$ direction are relevant in the classification. Also this makes sense, assuming that the gestures are drawn in the air roughly vertically, occupying mostly the $x$-$z$ plane. Similar explanations could also be produced for the JADE and principal component analysis solutions (not shown here).

\begin{table}[ht]
\centering
\caption{The loadings of the FOBI estimate. $\hat{Z}_j$ refer to the estimated components and $X_k$ to the original functions.}
\tabcolsep=0.11cm
\begin{tabular}{ccccccccccccc}
$\hat{Z}_j$ & $X_k$ & 1 & 2 & 3 & 4 & 5 & 6 & 7 & 8 & 9 & 10 & 11 \\
 & 1 & \cellcolor{gray!20} -0.90 & 0.36 & \cellcolor{gray!20} -1.21 & -0.13 & \cellcolor{gray!20} 0.79 & 0.17 & 0.19 & -0.01 & -0.03 & -0.06 & 0.04 \\
1 & 2 & -0.33 & -0.07 & -0.52 & -0.32 & 0.59 & -0.02 & 0.07 & -0.00 & -0.05 & -0.08 & 0.01 \\
 & 3 & -0.26 & \cellcolor{gray!20} 0.92 & -0.40 & \cellcolor{gray!20} -1.16 & 0.59 & -0.32 & 0.32 & 0.00 & -0.02 & -0.05 & -0.01 \\
 & 1 & 0.49 & \cellcolor{gray!20} 1.49 & -0.08 & \cellcolor{gray!20} -1.67 & -0.10 & -0.45 & -0.07 & -0.21 & 0.12 & -0.11 & 0.10 \\
2 & 2 & -0.16 & 0.30 & -0.04 & -0.09 & 0.13 & -0.16 & -0.07 & 0.01 & -0.04 & -0.02 & 0.06 \\
 & 3 & -0.20 & -0.24 & \cellcolor{gray!20} -0.86 & \cellcolor{gray!20} 0.83 & \cellcolor{gray!20} 0.86 & 0.27 & 0.52 & 0.05 & -0.17 & 0.12 & -0.11 \\
 & 1 & 0.16 & 0.39 & 0.31 & -0.43 & -0.26 & -0.42 & 0.22 & -0.00 & -0.14 & -0.04 & -0.14 \\
3 & 2 & \cellcolor{gray!20} 0.78 & 0.51 & -0.10 & -0.30 & -0.35 & -0.12 & 0.09 & -0.03 & 0.10 & 0.02 & -0.07 \\
 & 3 & 0.15 & 0.52 & \cellcolor{gray!20} 0.72 & -0.52 & \cellcolor{gray!20} -0.99 & -0.49 & -0.53 & -0.08 & 0.24 & -0.14 & 0.13 \\
\end{tabular}
\label{tab:real_1}
\end{table}

\section{Discussion}\label{sec:discussion}

We close the paper by discussing some directions for future research. First, while the provided rule of thumb of choosing $d = p$ proved useful in the examples, the logical next step is to provide a more analytical approach, e.g. in the form of sequential hypothesis testing.

Second, Theorem \ref{theo:unmixing} shows how the independent component scores are obtained but tells us nothing about the division of the scores into the independent subvectors. In our real data example this was not an issue as visual inspection already revealed us the scores of interest, but in the case of less visual data some kind of testing procedure is called for. A similar problem was encountered in \cite{nordhausen2016independent} where an approach based on scatter matrices with the independence property was used to identify the independent subvectors, and a likewise procedure could possibly also be used here.

Third, in Section \ref{sec:sample} it was shown that the extensions of both FOBI and JADE to multivariate functional data can be applied in practice by projecting the observed functions into the space spanned by the first $d$ eigenvectors of the covariance matrix operator and then subjecting the obtained standardized principal component coefficients to regular FOBI or JADE. This naturally begs for the question whether also some other standard multivariate methods can be meaningfully extended to multivariate functional data simply by applying them to the principal component coefficients. Some preliminary testing shows that this is certainly the case for FastICA, a projection pursuit-based family of independent component methods \citep{hyvarinen1997fast}.

%\bigskip
%\begin{center}
%{\large\bf SUPPLEMENTARY MATERIAL}
%\end{center}
%
%\begin{description}
%
%\item[The simulation and real data example codes:] An R-code file containing the codes for performing the simulations and the real data example of Section \ref{sec:examples}. (simulation\_codes.R, R file)
%
%\item[Appendix to ``Independent component analysis for multivariate functional data'':] An appendix containing the proofs of the results. (appendix.pdf, pdf file)
%
%\end{description}

\appendix

\section{Proofs of results}

%\hrule

%\medskip

\begin{proof}[Proof of Lemma \ref{lem:S_properties}]
The self-adjointness of $\Sigma_{XX}$ follows simply from the earlier discussion of the adjoints of the components $\Sigma \lo {X \lo i X \lo j}$. Furthermore, by expanding element-wise we have for any $f \in \mathcal{H}$:
\begin{equation*}
\langle \Sigma \lo {XX} f , f \rangle_\mathcal{H} = E \left\{ \langle \left( X \otimes X \right) f, f \rangle_\mathcal{H} \right\} = E \left( \langle X , f \rangle_\mathcal{H}^2 \right) \geq 0,
\end{equation*}
showing that $\Sigma \lo {XX}$ is non-negative.

Let $\{ e_k \}_{k=1}^\infty$ be an orthonormal basis of $\mathcal{H}$. Using the same reasoning as above, the trace of the self-adjoint, non-negative operator $\Sigma \lo {XX}$ is then
\begin{equation*}
\mbox{tr}(\Sigma \lo {XX}) = \sum_{k=1}^\infty \langle \Sigma \lo {XX} e_{k} , e_{k} \rangle_\mathcal{H} = E \left( \sum_{k=1}^\infty \langle X , e_k \rangle_\mathcal{H}^2  \right) = E \| X \|^2_\mathcal{H},
\end{equation*}
where the last equality uses Parseval's identity. Now, by our assumptions $E \| X \|^2_\mathcal{H}$ is finite, making $\Sigma_{X X}$ a trace-class operator.

To show that the affine equivariance holds, let $A \in \mathcal{L}(\mathcal{H})$ and write
\begin{equation*}
\Sigma \left( A X \right) = E \left( AX \otimes AX \right) = E \left\{ A \left( X \otimes X \right) A^* \right\}.
\end{equation*}
Thus $\Sigma \left( A X \right)$ is the unique operator $C$ satisfying $\langle f, C g \rangle_\mathcal{H} = E \left\{ \langle f, A \left( X \otimes X \right) A^* g \rangle_\mathcal{H} \right\}$, for all $f, g \in \mathcal{H}$. Using again the definition of the expected value of a random operator the right-hand side is seen to equal $\langle f, A \Sigma ( X ) A^* g \rangle_\mathcal{H}$ showing that $\Sigma \left( A X \right) = A \Sigma ( X ) A^*$.

Finally, the full independence property follows simply by assuming that $X_i$ and $X_j$ are independent and checking that we have
\begin{equation*}
E \left\{ \langle f_i, (X_i \otimes X_j) g_j \rangle_i \right\} = E \left( \langle X_i , f_i \rangle_i \right) E \left( \langle X_j , g_j \rangle_j \right) = 0,
\end{equation*}
for all $f_i \in \mathcal{H}_i$ and $g_j \in \mathcal{H}_j$, and thus by definition $E(X_i \otimes X_j) = 0$.
\end{proof}

\hrule

\begin{proof}[Proof of Lemma \ref{lem:standardization}]
Since $\Sigma$ is affine equivariant, and since   $\{ \phi_k \}_{k=1}^\infty$ are eigenvectors of $\Sigma(X)$, we have
\begin{equation*}
\Sigma(X^{(d)}) = P_{\mathcal{M}_d} \Sigma(X) P_{\mathcal{M}_d} = \sum_{k=1}^d \lambda_k (\phi_k \otimes \phi_k),
\end{equation*}
which further implies that $\Sigma(X^{(d)})^{-1/2} = \sum_{k=1}^d \lambda_k^{-1/2} (\phi_k \otimes \phi_k)$. Next, for $Z^{(d)} = \Gamma_0 X^{(d)}$ we have
\begin{equation*}
\Sigma(Z^{(d)}) = \sum_{k=1}^d \lambda_k (\Gamma_0 \phi_k \otimes \Gamma_ 0 \phi_k) = \Gamma_0 \Sigma(X^{(d)}) \Gamma_0^*.
\end{equation*}
As $\Gamma_0$ is boundedly invertible, the inverse square root of $\Sigma(Z^{(d)})$ exists as a bounded operator, and we can write
\begin{equation}\label{eq:standardization}
\Sigma(Z^{(d)})^{-1/2} Z^{(d)} = \left\{ \Sigma(Z^{(d)})^{-1/2} \Gamma_0 \Sigma(X^{(d)})^{1/2} \right\} \Sigma(X^{(d)})^{-1/2} X^{(d)}.
\end{equation}
What remains is to prove that $A_0 = \Sigma(Z^{(d)})^{-1/2} \Gamma_0 \Sigma(X^{(d)})^{1/2}$ is unitary which follows by directly verifying,
\begin{equation*}
A_0 A_0^* = \Sigma(Z^{(d)})^{-1/2} \Gamma_0 \Sigma(X^{(d)}) \Gamma_0^* \Sigma(Z^{(d)})^{-1/2},
\end{equation*}
where $\Gamma_0 \Sigma(X^{(d)}) \Gamma_0^*$ is equal to $\Sigma(Z^{(d)})$, showing that $A_0 A_0^* = P_{\mathcal{M}_d}$. The operator $A_0$ is thus unitary and consequently also $A_0^* A_0 = P_{\mathcal{M}_d}$. Applying now $A_0^*$ from left to both sides of \eqref{eq:standardization} shows that $U_0 = A_0^*$, concluding the proof.
\end{proof}

\hrule

\begin{proof}[Proof of Lemma \ref{lem:cumulant_operator_2}]
We provide the proof for the second term in \eqref{eq:cumulant_operator_1}, the proofs for the third and fourth terms following similarly. Using the definition of the expected value of a random operator, the second term is the unique operator $A \in \mathcal{L}(\mathcal{M}_d)$ with
\[\langle f, A g \rangle_\mathcal{H} = E \left( \langle X', \phi_i \rangle_\mathcal{H} \langle X', \phi_j \rangle_\mathcal{H} \langle X, f \rangle_\mathcal{H} \langle X, g \rangle_\mathcal{H} \right), \]
for all $f, g \in \mathcal{M}_d$. The independence of $X$ and $X'$ further implies that the right-hand side can be written in the form
\[E \left\{ \langle (X' \otimes X') \phi_i, \phi_j  \rangle_\mathcal{H} \right\} E \left\{ \langle (X \otimes X) f, g  \rangle_\mathcal{H} \right\} = \langle \Sigma(X') \phi_i, \phi_j  \rangle_\mathcal{H} \langle \Sigma(X) f, g  \rangle_\mathcal{H}, \]
which equals $\langle f, \delta_{ij} P_{\mathcal{M}_d} g \rangle$ under our assumptions, concluding the proof.
\end{proof}

\hrule

\begin{proof}[Proof of Theorem \ref{theo:Cij_properties}]
Consider only the first term in the expansion of $C^{ij}$ in Lemma \ref{lem:cumulant_operator_2}. Plugging in $UZ$ we get $U E \{ \langle U Z, \phi_i \rangle_\mathcal{H} \langle U Z, \phi_j \rangle_\mathcal{H} ( Z \otimes Z ) \} U^* = U M U^*$. The $(k, l)$ component operator of the expected value $M$ is defined as the operator $A_{kl} \in \mathcal{L}(\mathcal{H}_l, \mathcal{H}_k)$ satisfying
\begin{equation}\label{eq:Cij_expectation}
\langle f_k, A_{kl} g_l \rangle_k = \sum_{s, t = 1}^p E \left( \langle Z_s, \xi_{is} \rangle_s \langle Z_{t}, \xi_{jt} \rangle_{t} \langle Z_k, f_k \rangle_k \langle Z_l, g_l \rangle_l \right),
\end{equation}
for all $f_k \in \mbox{span}(\{ \phi_{mk} \}_{m=1}^d)$ and $g_l \in \mbox{span}(\{ \phi_{ml} \}_{m=1}^d)$ where $\xi_i = (\xi_{i1}, \ldots , \xi_{ip}) = U^* \phi_i$. Concentrate first on the off-diagonal case $k \neq l$. Then either $s = k, t = l$ or $s = l, t = k$ as otherwise the independence and zero means of the component functions reduce the sum to zero. Consider the first of these cases:
\begin{equation*}
E \left\{ \langle f_k, (Z_k \otimes Z_k)  \xi_{ik} \rangle_k \right\} E \left\{ \langle  g_l, (Z_l \otimes Z_l) \xi_{jl} \rangle_l \right\} =  \langle f_k , (\xi_{ik} \otimes \xi_{jl}) g_l \rangle_k.
\end{equation*}
The expected value of an arbitrary $(k, l)$th off-diagonal component operator of $M$ is thus $(\xi_{ik} \otimes \xi_{jl}) + (\xi_{jk} \otimes \xi_{il})$, which can be recognized to be also the $(k, l)$th component operator of $U^* \{ (\phi_i \otimes \phi_j) + (\phi_j \otimes \phi_i) \} U $.

The general form for an arbitrary $(k, k)$th diagonal component operator of $M$ can be found in a similar manner. Notice first that if $k=l$ in \eqref{eq:Cij_expectation} then it must be that $s = t$ or otherwise the sum is again zero by independence and zero means. The summation over $s$ can then be divided into two cases, $s = k$ and $s \neq k$. Similar manipulation as done above yields then the expected value $E \{ (Z_k \otimes Z_k) (\xi_{ik} \otimes \xi_{jk}) (Z_k \otimes Z_k) \} + (\delta_{ij} - \langle \xi_{ik}, \xi_{jk} \rangle_k) P_k$ for the $(k, k)$th diagonal operator where the first summand comes from the former case and the second from the latter.

Putting now everything together into a matrix of operators shows that the three last terms in the alternative form for $C^{ij}$ in Lemma \ref{lem:cumulant_operator_2} cancel out, leaving us with the claimed result.
\end{proof}

\hrule

\begin{proof}[Proof of Theorem \ref{theo:C_properties}]
For an arbitrary $X \in \mathcal{X}^4(\mathcal{M}_d)$ with $\Sigma(X) = P_{\mathcal{M}_d}$ we have by Lemma \ref{lem:cumulant_operator_2}
\begin{equation}\label{eq:C_simplified}
C(X) = \sum_{i=1}^d C^{ii}(X) = E \left\{ \sum_{i=1}^d \langle X, \phi_i \rangle_\mathcal{H}^2 ( X \otimes X ) \right\} - (d + 2) P_{\mathcal{M}_d},
\end{equation}
the argument of the expectation being further simplified by Parseval's identity to $\sum_{i=1}^d \langle X, \phi_i \rangle_\mathcal{H}^2 ( X \otimes X ) = \| X \|^2_\mathcal{H} ( X \otimes X ) = ( X \otimes X )^2$. The first claimed equality now follows from the form $C(X) = E \{ ( X \otimes X )^2 \} - (d + 2) P_{\mathcal{M}_d}$.

By Theorem \ref{theo:Cij_properties}, an arbitrary off-diagonal element of the operator $C(Z) = \sum_{i=1}^d D^{ii}$ is zero. The exact form for its diagonal elements $D_{kk}$ could also be derived from \eqref{eq:Cij_diagonal} but the seeming dependency of $D_{kk} = \sum_{i=1}^d D^{ii}_{kk}$ on the operator $U$ needlessly complicates things and it is simpler to proceed straight from the form $C(Z) = E[( Z \otimes Z )^2] - (d + 2) P_{\mathcal{M}_d}$. The $(k,k)$th diagonal operator of the first term is then defined as the unique operator $A_{kk} \in \mathcal{L}(\mathcal{H}_k, \mathcal{H}_k)$ satisfying
\begin{equation*}%\label{eq:C_expectation}
\langle f_k, A_{kk} g_k \rangle_k = E \left\{ \sum_{j=1}^d  \langle Z_j, Z_j \rangle_j \langle Z_k, f_k \rangle_k \langle Z_k, g_k \rangle_k \right\},
\end{equation*}
for all $f_k \in \mbox{span}(\{ \phi_{mk} \}_{m=1}^d)$ and $g_k \in \mbox{span}(\{ \phi_{mk} \}_{m=1}^d)$. Divide the summation over $j$ into two cases, $j = k$ and $j \neq k$. The former yields the term $E \{ \langle f_k, \| Z_k \|^2_k (Z_k \otimes Z_k), g_k \rangle_k \}$ contributing $E \{ \| Z_k \|^2_k (Z_k \otimes Z_k) \} = E \{ (Z_k \otimes Z_k)^2 \}$ to the final expected value. The latter yields the term
\begin{align*}
\left\{ \sum_{j \neq k} E \left( \| Z_j \|_j^2 \right) \right\} E \left\{ \langle f_k, \left( Z_k \otimes Z_k \right) , g_k \rangle_k \right\} = \left\{ \sum_{j \neq k} E \left( \| Z_j \|_j^2 \right) \right\} \langle f_k, g_k \rangle_k,
\end{align*}
where the first multiplicand can be written as $E \left( \| Z \|_\mathcal{H}^2 \right) - E \left( \| Z_k \|_k^2 \right)$, whose first term equals by Parseval's identity
\begin{equation*}
E \left( \| Z \|_\mathcal{H}^2 \right) = \sum_{i=1}^d E \left( \langle Z, \phi_i \rangle^2_\mathcal{H} \right) = \sum_{i=1}^d  \langle \phi_i, E \left( Z \otimes Z \right) \phi_i \rangle = d.
\end{equation*}
Similarly, by choosing an orthonormal basis for the $k$th component space one can show that $E \left( \| Z_k \|_k^2 \right) = d_k := \mbox{dim} [ \mbox{span}(\{ \phi_{mk} \}_{m=1}^d) ]$. The total contribution of the case $j \neq k$ to the expected value of the $(k, k)$th diagonal operator is thus $(d - d_k) P_k$. Finally, putting everything together with \eqref{eq:C_simplified} yields the desired result.
\end{proof}

\hrule
\begin{proof}[Proof of Lemma \ref{lem:diagonal_spectra}]
Inspect without loss of generality the first eigenvector $\psi_1$ and assume that it is not canonical, $\psi_1 = (\psi_{11} , \ldots , \psi_{1p})$, where again without loss of generality we assume that $\psi_{11}$ and $\psi_{12}$ are both non-zero. Then the linearly independent vectors $(\psi_{11}, 0, 0, \ldots , 0)$ and $(0, \psi_{12}, 0, \ldots , 0)$ are both eigenvectors of $D$ associated with the same eigenvalue $\tau_1$, making the eigenspace associated with the eigenvalue $\tau_1$ have dimension of at least 2, a contradiction as the assumption on the distinctness of eigenvalues implies unit rank. Thus only one of $\psi_{11} , \ldots , \psi_{1p}$ can be non-zero.
\end{proof}

\hrule
\begin{proof}[Proof of Lemma \ref{lem:joint_diag}]
By the Cauchy-Schwarz inequality and the unit length of $\psi_k$ we have
\begin{align*}
w \left( \psi_1, \ldots , \psi_d \right) \leq \sum_{i=1}^I \sum_{k=1}^d \langle \psi_k, \psi_k \rangle_{\mathcal{H}} \langle S_i \psi_k, S_i \psi_k \rangle_{\mathcal{H}} = \sum_{i=1}^I \sum_{k=1}^d  \langle S_i \psi_k, S_i \psi_k \rangle_{\mathcal{H}}.
\end{align*}
Now, $\sum_{k=1}^d \langle S_i \psi_k, S_i \psi_k \rangle_{\mathcal{H}} = \| S_i \|_{HS}^2$ for any orthonormal basis $\left\{ \psi_k \right\}_{k=1}^d$ and we have shown the first part of the claim. To see when the equality holds recall that the Cauchy-Schwarz inequality preserves equality if and only if the two vectors in question are proportional. We must thus have $\psi_k = a_{ik} S_i \psi_k$ for some $a_{ik} \in \mathbb{R}$ for all $i = 1, \ldots , I$, $k = 1, \ldots d$, which is equivalent to saying that each $\psi_k$ is an eigenvector of each $S_i$.
\end{proof}

\hrule

\begin{proof}[Proof of Theorem \ref{theo:unmixing}]
Recall first that by Lemma \ref{lem:standardization} we have $\tilde{X} = U_0 \tilde{Z}$ where $\tilde{Z} = \Sigma(Z^{(d)})^{-1/2} Z^{(d)}$. By Lemma \ref{lem:S_properties} the operator $\Sigma(Z^{(d)})$ is diagonal and thus one possible choice for the inverse square root of the operator $\Sigma(Z^{(d)})$ is also a diagonal operator, namely the diagonal operator $G$ with some inverse square roots of the diagonal elements of $\Sigma(Z^{(d)})$ as its diagonal elements. With this choice, $\Sigma(Z^{(d)})^{-1/2} = G$, also $\tilde{Z}$ has then independent component functions. A reasoning similar to the one used in Remark 2.1 in \cite{ilmonen2012invariant} shows that all inverse square roots of $\Sigma(Z^{(d)})$ are of the form $V G$ where $V$ is unitary and can by the unitary equivariance be taken out of $C(\tilde{Z})$, ``merging'' it with $U_0$. We may thus without loss of generality assume that $\Sigma(Z^{(d)})^{-1/2}$ is a diagonal operator. Invoking then finally Theorem \ref{theo:C_properties} shows that $C(\tilde{Z})$ is also a diagonal operator.

Let $\{ h^F_k \}_{k=1}^d$ be the eigenvectors of $C(\tilde{Z})$. Then by Theorem \ref{theo:C_properties} the FOBI-basis of $X$ is given by $\{ \psi^F_k \}_{k=1}^d = \{ U_0 h^F_k \}_{k=1}^d$. Then, by the unitarity of $U_0$ we have
\begin{equation*}
\hat{Z}^F_k = U_0 (h^F_{k} \otimes h^F_{k}) \tilde{Z}.
\end{equation*}
As $C(\tilde{X})$ and $C(\tilde{Z})$ share the same eigenvalues all the assumptions of Lemma \ref{lem:diagonal_spectra} are satisfied and only the $l(k)$th element of $h^F_k$ is non-zero, $k=1, \ldots d$. Consequently
\begin{equation*}
\hat{Z}_k = \langle h^F_{kl(k)}, \tilde{Z}_{l(k)} \rangle_{l(k)} U_0 h^F_{k},
\end{equation*}
showing that $\hat{Z}_k$ depends only on the $l(k)$th component of $Z$.

The result for the JADE-basis follows similarly. We first notice that by Theorem \ref{theo:Cij_properties} the operators $C^{ij}$ are semi-unitary equivariant in the sense that we may again assume that $\Sigma(Z^{(d)})^{-1/2}$ is a diagonal operator and that the random function $\tilde{Z}$ has independent component functions. Let then $\{ \psi^J_k \}_{k=1}^d$ be the joint diagonalizer of $\mathcal{C}$. Now, again by Theorem \ref{theo:Cij_properties} we have $C^{ij}(\tilde{X}) = \textbf{U}_0 D^{ij} \textbf{U}_0^*$ where $D^{ij} = D^{ij}(U_0, \tilde{Z})$ are diagonal operators, $i, j = 1, \ldots , d$. By Lemma \ref{lem:joint_diag} the joint diagonalizer of the set $\{ D^{ij} \}_{i,j = 1}^d$ is $\{ h^J_k \}_{k=1}^d$ where each $h^J_k$ is canonical. Consequently, the joint diagonalizer of $\mathcal{C}$ is $\{ \psi^J_k \}_{k=1}^d = \{ U_0 h^J_k \}_{k=1}^d$ and the desired result follows as above with FOBI.
\end{proof}

\hrule

\begin{proof}[Proof of Theorem \ref{theo:coordinates}]

First, our space being finite-dimensional, for every fixed pair of bases $\mathcal{B}, \mathcal{G}$ every linear operator $A$ in $\mathcal{M}$ has with it associated the unique matrix ${}_{\mathcal{G}}[A]_{\mathcal{B}} \in \mathbb{R}^{pK \times pK}$ that satisfies $[A f]_{\mathcal{G}} = ({}_{\mathcal{G}}[A]_{\mathcal{B}}) [f]_{\mathcal{B}}$, for all $f \in \mathcal{M}$. Furthermore, a function $f$ is an eigenfunction of the operator $A$ associated with the eigenvalue $\lambda$ if and only if $[f]_\mathcal{B}$ is an eigenvector of the matrix ${}_\mathcal{B}[A]_\mathcal{B}$ associated with the same eigenvalue $\lambda$.

The inner product of two elements $f_1, f_2 \in \mathcal{M}$ expressed in the same basis $\mathcal{G} = \{ g_k \}_{k=1}^K$ is given simply by
\begin{equation*}
\langle f_1 , f_2 \rangle = \sum_{k=1}^K \sum_{k'=1}^K ([f_1]_\mathcal{G})_k ([f_2]_\mathcal{G})_{k'} \langle g_k , g_{k'} \rangle = [f_1]_\mathcal{G}^T \textbf{G}_{\mathcal{G}} [f_2]_\mathcal{G},
\end{equation*}
where $\textbf{G}_\mathcal{G} = ( \langle g_k , g_{k'} \rangle )_{k,k'=1}^K$ is the Gram matrix of the basis $\mathcal{G}$. The tensor product between two elements $f \lo 1, f \lo 2 \in \ca M$ has the following coordinate
\begin{align}\label{eq:tensor product}
\lo {\ca G} [ f \lo 1  \otimes f \lo 2 ] \lo {\ca G} = [f \lo 1 ] \lo {\ca G} [ f \lo 2 ] \lo {\ca G} ^T \textbf{G} \lo {\ca G}.
\end{align}
These and more properties about the coordinate system were used and further developed in \cite{li2016}.

We begin with the coordinate representation of the standardization step. An estimate for the covariance matrix operator is
\begin{equation*}
\hat{\Sigma} \lo {X \lo r X \lo s} = \frac{1}{n} \sum_{i=1}^n (X_{ir} \otimes X_{is}).
\end{equation*}
The coordinate of $\hat \Sigma \lo {XX}$ is the matrix $\{ \lo {\ca G \lo 0} [\hat \Sigma \lo  {X \lo r X \lo s} ] \lo {\ca G \lo 0} \} \lo {r,s=1} \hi p$  By (\ref{eq:tensor product}),
\begin{align*}
\lo {\ca G \lo 0  } [ \hat  \Sigma \lo {X \lo r X \lo s} ] \lo {\ca G \lo 0 } = \frac{1}{n} \sum \lo {i=1} \hi n [ X \lo {ir} ]  \lo {\ca G \lo 0 } [ X \lo {is}  ] \lo {\ca G \lo 0} ^T \textbf{G} \lo {\ca G \lo 0}.
\end{align*}
Assemble these matrices together to obtain
\begin{align*}
\lo {\ca G} [ \hat \Sigma \lo {XX} ] \lo {\ca G} = \frac{1}{n} [X] \lo {\ca G} ^T [X] \lo {\ca G} (\textbf{I} \lo p \otimes \textbf{G} \lo {\ca G _0}),
\end{align*}
where $\otimes$ is the Kronecker product between matrices.

We next fix the dimension $d \leq pK$ and estimate the coordinate $[\hat{\phi}_l] \lo {\ca G}$ of the first $d$ eigenfunctions $\hat \phi \lo l$ of $\hat \Sigma \lo {XX}$. As shown in \cite{li2016}, $\hat \phi \lo l$ is the $l$th eigenfunction of the operator  $\hat \Sigma \lo {XX}$ if and only if $(\textbf{I} \lo p \otimes \textbf{G} \lo {\ca G _0} \hi {1/2} )[\hat \phi \lo l] \lo {\ca G}$ is the $l$th eigenvector of the matrix  $(\textbf{I} \lo p \otimes \textbf{G} \lo {\ca G _0} \hi {1/2} ) ( \lo {\ca G} [\hat \Sigma \lo {XX} ] \lo {\ca G} ) (\textbf{I} \lo p \otimes \textbf{G} \lo {\ca G _0} \hi {-1/2}$). The orthogonal projection of $f \in \mathcal{M}$ onto $span(\mathcal{V})$, where $\mathcal{V} = \{ \hat{\phi}_l \}_{l=1}^d$, is then
\begin{equation*}
\sum_{l=1}^d (\hat{\phi}_l \otimes \hat{\phi}_l) f = \sum_{k=1}^d [\hat{\phi}_{l}] \lo {\ca G}^T (\textbf{I} \lo p \otimes \textbf{G} \lo {\ca G _0} ) [f] \lo {\ca G} \,  \hat{\phi}_l,
\end{equation*}
and the coordinates of the observations in the eigenbasis $\mathcal{V}$ are thus
\begin{align*}
[X \hi {i (d)}]\lo {\ca V} = [\hat{\Phi}]^T \lo {\ca G} (\textbf{I} \lo p \otimes \textbf{G} \lo {\ca G _0}) [X \hi i] \lo {\ca G}  \in \mathbb{R}^d, \ \, i=1, \ldots , n,
\end{align*}
where $[\hat{\Phi}]\lo {\ca G} = ([\hat{\phi}_1]\lo {\ca G}, \ldots , [\hat{\phi}_d]\lo {\ca G})$. Let $[X \hi {(d)}] \lo {\ca V}=
([X \hi {1(d)}] \lo {\ca V}, \ldots, [X \hi {n(d)}] \lo {\ca V} )^T$. Then the above equations can be written in matrix form as
\begin{align*}
[X^{(d)}]\lo {\ca V} = [X]\lo {\ca G} (\textbf{I} \lo p \otimes  \textbf{G} \lo {\ca G _0}) [\hat{\Phi}]\lo {\ca G}.
\end{align*}
Since the principal component scores satisfy ${}\lo {\ca V} [ \hat{\Sigma}(X^{(d)})]\lo {\ca V} = \boldsymbol{\Lambda}_d$, where $\boldsymbol{\Lambda}_d = \mbox{diag}(\lambda_1, \ldots , \lambda_d)$ contains the eigenvalues of $\hat \Sigma \lo {XX}$, the coordinates of the standardized observations $\tilde X \hi i$ in the eigenbasis are
\begin{equation*}
[\tilde{X} \hi i ]\lo {\ca V} = ({}\lo {\ca V} [ \hat{\Sigma}(X^{(d)})]\lo {\ca V})^{-1/2} [X^{i(d)}]\lo {\ca V} = \boldsymbol{\Lambda}_d^{-1/2} [\hat{\Phi}]^T \lo {\ca G} (\textbf{I} \lo p \otimes \textbf{G} \lo {\ca G _0} ) [X \hi i] \lo {\ca G}.
\end{equation*}

Turning our attention to the fourth cross-cumulant operators we have for fixed $k,l = 1, \ldots, d$ the estimate
\begin{equation*}
\hat{C}^{kl}(\tilde{X}) = \frac{1}{n} \sum_{i=1}^n \langle \tilde{X} \hi i, \hat{\phi}_k \rangle\lo {\ca V} \langle \tilde{X} \hi i, \hat{\phi}_l \rangle\lo {\ca V} (\tilde{X} \hi i \otimes \tilde{X} \hi i)  - \delta_{kl} P_{\mathcal{M}_d} - \hat{\phi}_k \otimes \hat{\phi}_l - \hat{\phi}_l \otimes \hat{\phi}_k,
\end{equation*}
where the inner product $\langle \tilde{X } \hi i, \hat{\phi}_k \rangle\lo {\ca V}$ just extracts the $k$th element of the coordinate vector $[\tilde{X} \hi i]\lo {\ca V}$. Reasoning then as above with the covariance matrix operator it is straightforward to obtain the following coordinate representation:
\begin{equation*}
{}\lo {\ca V}[\hat{C}^{kl}(\tilde{X})]\lo {\ca V} = \frac{1}{n} \sum_{i=1}^n ([\tilde X \hi i]\lo {\ca V}^T \textbf{e}_k) ([\tilde X \hi i]\lo {\ca V}^T \textbf{e}_l) \cdot  [\tilde X \hi i]\lo {\ca V} [\tilde X \hi i]\lo {\ca V}^T  - \delta_{kl} \textbf{I}_d - \textbf{e}_k \textbf{e}_l^T - \textbf{e}_l \textbf{e}_k^T,
\end{equation*}
where $\textbf{e}_k$ is the $k$th canonical basis vector of $\mathbb{R}^d$ and $\textbf{I}_d$ is the $d \times d$ identity matrix. The similarity of this form to \eqref{eq:cumulant_matrix} already suggests that the functional independent component analysis solutions are found by performing regular FOBI or JADE on the coordinates $[\tilde X \hi i]\lo {\ca V}$ of the standardized observations.

The coordinate representation of the estimate of the FOBI-operator \eqref{eq:C_definition} is now simply
\begin{equation*}
{}\lo {\ca V}[\hat{C}(\tilde{X})]\lo {\ca V} = \frac{1}{n} \sum_{i=1}^n [\tilde X \hi i]\lo {\ca V}^T [\tilde X \hi i]\lo {\ca V} \cdot  [\tilde X \hi i]\lo {\ca V} [\tilde X \hi i]\lo {\ca V}^T  - (d + 2) \textbf{I}_d,
\end{equation*}
and an estimate $\mathcal{U}^F = \{ \hat{\psi}^F_m \}_{m=1}^d$ for the FOBI-basis is found from its eigendecomposition. Letting $[\hat{\Psi}^F]\lo {\ca V} = ([\hat{\psi}^F_1]\lo {\ca V}, \ldots , [\hat{\psi}^F_d]\lo {\ca V}) \in \mathbb{R}^{d \times d}$ be the coordinate representation of the eigenvectors of $\hat{C}(\tilde{X})$ in $\mathcal{V}$, the vector of the FOBI independent component scores, $\langle \hat{\psi}^F_m, \tilde X \hi i \rangle = [\hat{\psi}^F_m]\lo {\ca V}^T [\tilde X \hi i]\lo {\ca V}$, $m = 1, \ldots , d$, is then finally obtained as
\begin{equation*}
[\hat{\Psi}^F]^T\lo {\ca V} [\tilde X \hi i]\lo {\ca V} = [\hat{\Psi}^F]^T\lo {\ca V} \boldsymbol{\Lambda}_d^{-1/2} [\hat{\Phi}]^T\lo {\ca G} (\textbf{I} \lo p \otimes \textbf{G} \lo {\ca G _0}) [X \hi i ]\lo {\ca G}.
\end{equation*}

For the JADE-solution, an estimate $\mathcal{U}^J = \{ \hat{\psi}^J_m \}_{m=1}^d$ for the JADE-basis, i.e. the joint diagonalizer of the set $\{ \hat{C}^{kl}(\tilde{X}) \}_{k,l=1}^d$, is found by maximizing the quantity \eqref{eq:joint_diag}, the maximization problem now having the form
\begin{equation*}
[\hat{\Psi}^J]\lo {\ca V} = \underset{[\hat{\Psi}^J]\lo {\ca V}^T [\hat{\Psi}^J]\lo {\ca V} = \textbf{I}_d}{argmax}\sum_{k=1}^d \sum_{l=1}^d \sum_{m=1}^d \left\{ [\hat{\psi}^J_m]^T\lo {\ca V} ({}\lo {\ca V}[\hat{C}^{kl}(\tilde{X})]\lo {\ca V})  [\hat{\psi}^J_m]\lo {\ca V}  \right\}^2,
\end{equation*}
where $[\hat{\Psi}^J]\lo {\ca V} = ([\hat{\psi}^J_1]\lo {\ca V}, \ldots , [\hat{\psi}^J_d]\lo {\ca V}) \in \mathbb{R}^{d \times d}$. As with FOBI above, the vectors of the JADE independent component scores are then
\begin{equation*}
[\hat{\Psi}^J]^T\lo {\ca V} [\tilde X \hi i]\lo {\ca V} = [\hat{\Psi}^J]^T\lo {\ca V} \boldsymbol{\Lambda}_d^{-1/2} [\hat{\Phi}]^T\lo {\ca G} (\textbf{I} \lo p \otimes \textbf{G} \lo {\ca G _0} ) [X \hi i]\lo {\ca G}.
\end{equation*}

\end{proof}

\bibliographystyle{Chicago}

\bibliography{references}

\end{document}